\documentclass[reqno,11pt]{amsart}
\usepackage{amsmath,amssymb,amsfonts,amsthm, amscd,indentfirst}
\usepackage{amsmath,latexsym,amssymb,amsmath,
	amscd,amsthm,amsxtra}
\usepackage{color}
\usepackage{pdfpages}
\usepackage{graphics}
\usepackage{enumitem}

\usepackage[usenames,dvipsnames]{xcolor}

\usepackage{float}
\usepackage{varioref}
\usepackage{hyperref}
\usepackage[noabbrev,capitalise,nameinlink]{cleveref}

\PassOptionsToPackage{unicode}{hyperref}
\PassOptionsToPackage{naturalnames}{hyperref}

\hypersetup{colorlinks={true},linkcolor={blue},citecolor=blue}
\usepackage[english]{babel}
\usepackage{csquotes}

\usepackage[backend=biber,
style=alphabetic,datamodel=ext-eprint,maxbibnames=99]{biblatex}
\DeclareSourcemap{
	\maps[datatype=bibtex]{
		\map{
			\step[fieldsource=pmid, fieldtarget=pubmed]
		}
	}
}
\makeatletter
\DeclareFieldFormat{arxiv}{%
	arXiv\addcolon\space
	\ifhyperref
	{\href{http://arxiv.org/\abx@arxivpath/#1}{%
			\nolinkurl{#1}%
			\iffieldundef{arxivclass}
			{}
			{\addspace\texttt{\mkbibbrackets{\thefield{arxivclass}}}}}}
	{\nolinkurl{#1}
		\iffieldundef{arxivclass}
		{}
		{\addspace\texttt{\mkbibbrackets{\thefield{arxivclass}}}}}}
\makeatother
\DeclareFieldFormat{pmcid}{%
	PMCID\addcolon\space
	\ifhyperref
	{\href{http://www.ncbi.nlm.nih.gov/pmc/articles/#1}{\nolinkurl{#1}}}
	{\nolinkurl{#1}}}
\DeclareFieldFormat{mr}{%
	MR\addcolon\space
	\ifhyperref
	{\href{http://www.ams.org/mathscinet-getitem?mr=MR#1}{\nolinkurl{#1}}}
	{\nolinkurl{#1}}}
\DeclareFieldFormat{zbl}{%
	Zbl\addcolon\space
	\ifhyperref
	{\href{http://zbmath.org/?q=an:#1}{\nolinkurl{#1}}}
	{\nolinkurl{#1}}}
\DeclareFieldAlias{jstor}{eprint:jstor}
\DeclareFieldAlias{hdl}{eprint:hdl}
\DeclareFieldAlias{pubmed}{eprint:pubmed}
\DeclareFieldAlias{googlebooks}{eprint:googlebooks}

\renewbibmacro*{eprint}{%
	\printfield{arxiv}%
	\newunit\newblock
	\printfield{jstor}%
	\newunit\newblock
	\printfield{mr}%
	\newunit\newblock
	\printfield{zbl}%
	\newunit\newblock
	\printfield{hdl}%
	\newunit\newblock
	\printfield{pubmed}%
	\newunit\newblock
	\printfield{pmcid}%
	\newunit\newblock
	\printfield{googlebooks}%
	\newunit\newblock
	\iffieldundef{eprinttype}
	{\printfield{eprint}}
	{\printfield[eprint:\strfield{eprinttype}]{eprint}}}

\usepackage{graphicx}
\usepackage{caption}

\usepackage{epsfig,here}
\usepackage{subfigure,here}

\newtheorem{theorem}{Theorem}[section]
\newtheorem{lemma}[theorem]{Lemma}
\newtheorem{proposition}[theorem]{Proposition}
\newtheorem{definition}[theorem]{Definition}
\newtheorem{corollary}[theorem]{Corollary}
\newtheorem{remark}[theorem]{Remark}

\newtheorem{thmx}{Theorem}

\def\de{\delta}
\def\De{\Delta}
\def\La{\Lambda}
\def\ep{\epsilon}

\def\om{\omega}
\def\Om{\Omega}

\def\p{\partial}

\newcommand{\vvert}{\vert\vert}

\newcommand{\mfH}{\mathbf{H}}

\newcommand{\mfR}{\mathbf{R}}
\newcommand{\mfS}{\mathbf{S}}

\newcommand{\bsmu}{\boldsymbol\mu}
\newcommand{\bstau}{\boldsymbol\tau}
\newcommand{\bseta}{\boldsymbol\eta}

\newcommand{\mcA}{\mathcal{A}}

\newcommand{\mcG}{\mathcal{G}}
\newcommand{\mcH}{\mathcal{H}}

\newcommand{\mcL}{\mathcal{L}}

\newcommand{\mcU}{\mathcal{U}}

\newcommand{\ra}{\rightarrow}
\newcommand{\pr}{\partial_{\rm rel}}
\newcommand{\rd}{{\rm d}}

\numberwithin{equation} {section}

\begin{document}
	
	\title[Alexandrov-type theorem]{Alexandrov-type theorem for singular capillary CMC hypersurfaces in the half-space}
	
	\author[Xia]{Chao Xia}
	\address[Chao Xia]{
	\newline\indent School of Mathematical Sciences, Xiamen University,
	361005, Xiamen, P.R. China}
	\email{chaoxia@xmu.edu.cn}
	
	\author[Zhang]{Xuwen Zhang}
	\address[Xuwen Zhang]{
	\newline\indent School of Mathematical Sciences, Xiamen University,
	361005, Xiamen, P.R. China
\newline\indent\&
\newline\indent Institut f\"ur Mathematik, Goethe-Universit\"at, 
60325, Frankfurt, Germany}
	\email{xuwenzhang@stu.xmu.edu.cn}
 
\thanks{This work is  supported by the NSFC (Grant No. 11871406, 12271449).}
	
	\begin{abstract}
In this paper, we consider the classification problem for critical points of relative isoperimetric-type problem in the half-space. Under certain regularity assumption, we prove an Alexandrov-type theorem for the singular capillary CMC hypersurfaces in the half-space. The key ingredient is a new shifted distance function that is suitable for the study of capillary problem in the half-space.

		\
		
		
	\end{abstract}
	
	\maketitle
	
	\medskip

\section{Introduction}

The celebrated Alexandrov's theorem \cite{Aleksandrov62} in differential geometry says that any embedded closed constant mean curvature (CMC) hypersurface in the Euclidean space is a round sphere. Alexandrov developed the moving plane method to prove his theorem.
Ros \cite{Ros87} and Montiel-Ros \cite{MR91} found an alternative way to achieve Alexandrov's theorem, via Heintze-Karcher's inequality.
It is well-known that CMC hypersurfaces play the role as critical points of the Euclidean isoperimetric problem among $C^2$-hypersurfaces. 
From the perspective of modern calculus of variations, De Giorgi \cite{DeG58} has characterized round balls as the only isoperimetric sets among sets of finite perimeter. 
It is a natural question to characterize the critical points of the Euclidean isoperimetric problem among sets of finite perimeter. Quite recently, Delgadino-Maggi \cite{DM19} gave a complete characterization.

\begin{thmx}[{\cite[Theorem 1]{DM19}}]\label{Thm-DM19-1}
    Among sets of ﬁnite perimeter and ﬁnite volume, ﬁnite unions of balls with equal radii are the only critical points of the Euclidean isoperimetric problem.
\end{thmx}

It is known that if a set of finite perimeter and ﬁnite volume $\Om\subset\mfR^{n+1}$ is a critical point of the Euclidean isoperimetric problem, then up to a $\mcL^{n+1}$-negligible set, its topological boundary $\p\Om=\p^*\Om\cup (\p \Om\setminus \p^*\Om)$, where the reduced boundary $\p^*\Om$ is locally an analytic CMC hypersurface and relatively open in $\p\Om$, while $\mcH^n(\p \Om\setminus \p^*\Om)=0$, see for example \cite[subsection 2.4]{DM19}. 
In fact, Delgadino-Maggi \cite{DM19} proved that 
a set of finite perimeter and ﬁnite volume $\Om$ satisfying that the induced varifold of $\p^*\Om$ is of constant generalized mean curvature and $\mcH^n(\p \Om\setminus \p^*\Om)=0$, must be a ﬁnite union of balls with equal radii.
Delgadino-Maggi obtain their result by the subtle analysis which generalizes Montiel-Ros' argument in \cite{MR91} to sets of finite perimeter. We mention that similar consideration as Delgadino-Maggi has been also done by De Rosa-Kolasinski-Santilli \cite{DeRKS20} for the anisotropic case and by Maggi-Santilli \cite{MS23} concerning CMC in Brendle's class of warped product manifolds \cite{Bre13}.

The capillary phenomena appear naturally in the study of the equilibrium shape of liquid pendant drops and crystals in a given solid container.
The mathematical model has been established through the work of Young, Laplace,  Gauss and others,  as a variational problem on minimizing a free energy functional under volume constraint.
We are interested in a simple but important model, where the interior and the boundary of the container are both Euclidean, that is, the capillary phenomena in a Euclidean half-space. Let $\mfR^{n+1}_+=\{x\in\mfR^{n+1}: x\cdot E_{n+1}>0\}$ be the open upper half-space, where $E_{n+1}$ is the $(n+1)$-coordinate unit vector.
The global volume-constraint  minimizers for the corresponding relative isoperimetric-type problem  in $\mfR^{n+1}_+$
has been classified by Gonzalez \cite{Gonzalez76}. Precisely, for a set of finite perimeter and ﬁnite volume $\Om\subset \mfR^{n+1}_+$ and $\theta\in(0,\pi)$, consider the free energy functional 
$$\mathcal{E}(\Om)=P(\Om; \mfR^{n+1}_+)-\cos\theta P(\Om; \p\mfR^{n+1}_+).$$
Gonzalez \cite{Gonzalez76} proved the axially symmetric property of the global minimizers using the Schawarz symmetrization. A standard comparison argument by the isoperimetric inequality leads to the classification that the only volume-constraint global minimizers are spherical caps intersecting $\p\mfR^{n+1}_+$ at the angle $\theta$  see for example \cite[Section 2.2]{CM07-2} and \cite{CM07-1}.
See also \cite[Section 19.4]{Mag12} for an overview of the problem and  \cite{MM16} concerning the appearance of gravitational energy.
Recently, by some adaptions of the method proposed in \cite{SZ98} together with the regularity issue addressed by De Philippis-Maggi in \cite{DePM15,DePM17}, the uniqueness of the volume-constraint local minimizers of the free energy functional in the half-space has been characterized by the authors in \cite[Theorem 1.11]{XZ21}.

In the smooth setting, capillary CMC hypersurfaces play the role as critical points of the relative isoperimetric problem among $C^2$-hypersurfaces.
Here a capillary hypersurface in a container is the hypersurface that intersects the boundary of the container at a constant contact angle.

 Wente \cite{Wente80} exploited the moving plane method to prove an Alexandrov-type theorem, which says that any embedded capillary CMC hypersurface in  $\mfR^{n+1}_+$ must be a spherical cap.
Recently, joint with Jia and Wang \cite{JWXZ22-B}, we reprove Wente's result by developing a Heintze-Karcher-type inequality for capillary hypersurfaces in $\mfR^{n+1}_+$ in the spirit of \cite{MR91}. See also \cite{DW22} for a related consideration in the half-space and \cite{JWXZ22-A} for the anisotropic case.

 In the non-smooth setting, the study of CMC hypersurfaces has attached well attention. Using the min-max theory, Zhou-Zhu \cite{ZZ19} proved the existence of non-trivial, smooth, closed, almost embedded CMC hypersurfaces in any closed Riemmanian manifold $M^{n+1}(3\leq n+1\leq7)$, and then the result is extended to prescribed mean curvature (PMC) by the same authors in \cite{ZZ20}.
Very recently, the Min-Max method is used independently by De Masi-De Philippis \cite{DeMDeP21} and Li-Zhou-Zhu \cite{LZZ21} to show the existence of capillary minimal or CMC hypersurfaces in compact 3-manifolds with boundary.

Following \cite{LZZ21}, to study the capillary phenomenon in the non-smooth setting, we consider the following functional defined on sets of finite perimeter in the half-space.
 
\begin{definition}[$\mcA$-functional]\label{Defn-A-functional}
\normalfont
Given $\theta\in(0,\pi)$ and a constant $c>0$, for a set of finite perimeter and finite volume $\Om\subset\overline{\mfR^{n+1}_+}$, the $\mcA$-functional of $\Om$ with respect to $\theta$ and $c$ is given by
\begin{align}
    \mcA(\Om)=\mathcal{E}(\Om)-c\vert\Om\vert.
\end{align}

We say that $\Om$ is stationary for the $\mcA$-functional if for any $C^1$-diffeomorphism $f:\overline{\mfR^{n+1}_+}\ra\overline{\mfR^{n+1}_+}$ with compact support, such that $f:\p\mfR^{n+1}_+\ra\p\mfR^{n+1}_+$ is a diffeomorphism of $\p\mfR^{n+1}_+$,
there holds
\begin{align*}
    \frac{\rd}{\rd t}\mid_{t=0}\mcA(\psi_t(\Om))=0,
\end{align*}
where $\{\psi_t\}_{\vert t\vert<\epsilon}$ is a one parameter family of diffeomorphisms induced by $f$.
\end{definition}

\begin{definition}
    \normalfont
    For any bounded,
relatively open set of finite perimeter $\Om\subset\overline{\mfR^{n+1}_+}$,
Let $\pr\Om
:=\overline{\p\Om\cap\mfR^{n+1}_+}$ be the relative boundary of $\Om$ and $\Gamma:=\pr\Om\cap\p\mfR^{n+1}_+$.
    The \textit{regular part of $\pr\Om$} is defined by
    \begin{align*}
        {\rm reg}\pr\Om=\{x\in \pr\Om:\text{ there exists an }r_x>0\text{ such that }{\rm reg}\pr\Om\cap B_{r_x}(x)\\
        \text{is a }C^2\text{-manifold possibly with boundary contained in }\p\mfR^{n+1}_+\},
    \end{align*}
    while ${\rm sing}\pr\Om=\pr\Om\setminus{\rm reg}\pr\Om$ is called the \textit{singular set of $\pr\Om$}.
    In this way, ${\rm sing}\pr\Om$ is relatively closed in $\pr\Om$.
    
    We also denote by ${\rm reg}\Gamma:={\rm reg}\pr\Om\cap\Gamma$ the regular part of $M$ in $\Gamma$ and ${\rm sing}\Gamma=\Gamma\setminus{\rm reg}\Gamma$ the singular part of $M$ in $\Gamma$.
Moreover, $\nu_\Om$ denotes the outer unit normal of $\Om$ along ${\rm reg}\pr\Om$, $H$ denotes the mean curvature of ${\rm reg}\pr\Om$ in $\mfR^{n+1}$;
$\mu,\bar\nu$ denote the outer unit conormals of ${\rm reg}\Gamma$ in $\pr\Om$ and ${\rm reg}\Gamma$ in $T$, respectively. We refer to \cref{figure1} for illustration.
\end{definition}

Motivated by Delgadino-Maggi's work \cite{DM19}, a natural question is to characterize the critical points of the $\mcA$-functional among sets of finite perimeter. This is the main purpose of this paper.
\begin{theorem}\label{Thm-Alexandrov}
    Given $\theta\in(0,\pi)$ and $c>0$.
    Let $\Om\subset \overline{\mfR^{n+1}_+}$ be a bounded, relatively open set of finite perimeter and finite volume, which is stationary for the $\mcA$-functional.
    Assume that $\mcH^{n-1}({\rm sing}\pr\Om)=0$,
and  $\Gamma$ is a smooth $(n-1)$-manifold in $\p\mfR^{n+1}_+$.
    Then $\pr\Om$ must be a finite union of $\theta$-caps and spheres with equal radii.
\end{theorem}
\begin{remark}\label{Rem-assumption-mainThm}
\normalfont
    We make several remarks on the assumptions of our main theorem.
\begin{enumerate}
    \item 
    For the case $\theta=\pi/2$, the $\mcA$-stationary set is indeed a critical point of the relative isoperimetric problem (for area functional) in the half-space, and the characterization can be deduced from Delgadino-Maggi's result.
    Precisely,  provided $\mcH^n({\rm sing}\pr\Om)=0$, $\mcH^{n-1}({\rm sing}\Gamma)=0$,   Allard's regularity theorem for free boundary rectifiable varifolds \cite[Theorem 4.13]{GJ86} implies the young's law for $\theta=\pi/2$. Hence one may reflect $\pr\Om$ across the hyperplane $\{x_{n+1}=0\}$ and obtain a closed hypersurface $\tilde\Sigma$ such that the induced varifold of ${\rm reg}\tilde\Sigma$ if of constant generalized mean curvature and $\mcH^n({\rm sing}\tilde\Sigma)=0$, the assertion then follows from \cite[Theorem 1]{DM19}. 

    In view of the above,  we expect that for general $\theta$, the condition $\mcH^{n-1}({\rm sing}\pr\Om)=0$ in \cref{Thm-Alexandrov} mighted be weakened to $\mcH^n({\rm sing}\pr\Om)=0$, $\mcH^{n-1}({\rm sing}\Gamma)=0$.
Our assumption $\mcH^{n-1}({\rm sing}\pr\Om)=0$ is technical but crucial for the proof, we delay a detailed illustration to \cref{Rem-natural-assumption}.

    \item 
    The classical Alexandrov's moving plane method has been extended to the context of integral varifolds by Haslhofer-Hershkovits-White \cite{HHW20}, where they 
   made a so-called   tameness assumption on integral varifolds \cite[Definition 1.6]{HHW20}. The condition that $\Gamma$ is a smooth $(n-1)$-manifold is in fact included in the tameness condition.
  On the other hand, the assumption that $\Gamma$ is a smooth $(n-1)$-manifold and that $\pr\Om$ is regular enough up to $\Gamma$ for $\mcH^{n-1}$-a.e. ensures that the Young's law holds (which provides the contact angle condition).
    Such $\pr\Om$ will be defined as the singular capillary CMC hypersurface in the half-space in \cref{Defn-CapillaryCMC}.

\end{enumerate}
\end{remark}

To illustrate the proof, we first make a quick review of Jia-Wang-Xia-Zhang's argument  \cite{JWXZ22-B} in the smooth setting, which can be viewed as the extension of Montiel-Ros' argument \cite{MR91} to the capillary case. For $\Om\subset \overline{\mfR^{n+1}_+}$ such that $\pr\Om$ is of CMC and $\pr\Om$ intersects $\p\mfR^{n+1}_+$ at the angle $\theta$, we define a set $$Z=\left\{(x,t)\in \pr\Om\times \mfR: 0<t\le \frac{1}{\max_i \kappa_i(x)}\right\},$$ and a map  which indicates a family of shifted parallel hypersurfaces, $$\zeta_\theta: Z\to \mfR^{n+1}: \quad \zeta_\theta(x, t)=x-t(\nu_\Om-\cos\theta E_{n+1})$$
where $\kappa_i, i=1\cdots, n$ are the principal curvatures and $\max_i \kappa_i(x)\ge\frac{H_\Om}{n}>0$. Using the capillary boundary condition, we find that $\zeta_\theta$ is surjective onto $\Omega$, namely, $\Om\subset \zeta_\theta(Z)$. 
By the area formula, we obtain that
\begin{align*}
    |\Om|\le |\zeta_\theta(Z)|
    &\le \int_{\pr\Om}\int_0^\frac{1}{\max_i \kappa_i(x)}(1-\cos\theta \nu_\Om\cdot E_{n+1})\prod_{i=1}^n(1-t\kappa_i(x))\rd t\rd\mcH^n(x)\\
    &\le\int_{\pr\Om}\frac{1-\cos\theta \nu_\Om\cdot E_{n+1}}{H_\Om}\rd\mcH^n,
\end{align*}
so that the Heintze-Karcher inequality holds 
\begin{align*}
    |\Om|\le\int_{\pr\Om}\frac{1-\cos\theta\nu_\Om\cdot E_{n+1}}{H_\Om}\rd\mcH^n
\end{align*}
with equality holds if and only if $\pr\Om$ is a $\theta$-cap. 
Combining with the Minkowski-type formula
\begin{align*}
    \int_{\pr\Om}x\cdot(H_\Om\nu_\Om)\rd\mcH^n
    =\int_{\pr\Om}n(1-\cos\theta \nu_\Om\cdot E_{n+1})\rd\mcH^n,
\end{align*} we conclude that equality in Heintze-Karcher inequality holds and in turn, $\pr\Om$ is a $\theta$-cap.

Our aim is to generalize the above argument to sets of finite perimeter, in the spirit of Delgadino-Maggi \cite{DM19}. 
The key ingredient of Delgadino-Maggi's proof \cite{DM19} is to construct a large subset $\Om^\star$ of good points for a set of finite perimeter $\Om\subset\mfR^{n+1}$, with the property that:
for the classical parallel hypersurfaces map $\zeta$ considered in \cite{MR91} when restricted to the reduced boundary of $\Om$, one may show that
$|\Om^\star\setminus \zeta(Z)|=0$ and $|\Om\setminus \Om^\star|=0$, by virtue of which the argument based on the area formula is still applicable.
We point out that their construction of $\Om^\star$ is based on a subtle analysis of level-sets of the distance function from boundary, and to adapt their argument to our situation, the first requisite would be to discover a suitable capillary counterpart of the distance function considered in \cite{DM19}. This is done in light of the following observation.

In the proof of Jia-Wang-Xia-Zhang \cite{JWXZ22-B}, to show the surjectivity of $\zeta_\theta$, we have to show that for any $y\in\Om$, there is always some $x\in\Sigma$ such that $y$ can be flowed from $x$ through $\zeta_\theta$.
To this end, we consider the following foliation $\{\p B_r(y-r\cos\theta E_{n+1})\}_{r\geq0}$, and we are concerned with the first touching point of the foliation with $\Sigma$, when the radius increases from $0$.
The first touching point in this case somehow serves as a shifted `unique point projection' from $y$ to $\Sigma$, which motivates the following definition of
 shifted distance function.
Given a bounded, relatively open set of finite perimeter $\Om\subset\overline{\mfR^{n+1}_+}$ and $\theta\in(0,\pi)$,
let $\de:\mfR^{n+1}\to \mfR$ be the \textit{distance function with respect to $\pr\Om$}, defined as
\begin{align}
    \de(y)=\sup_{r\geq0}\{r:B_r(y)\cap\pr\Om=\varnothing\}.
\end{align}
and $\de_{\theta}:\mfR^{n+1}\to \mfR$ be the \textit{shifted distance function with respect to $\pr\Om$ and $\theta$}, defined as
\begin{align}\label{defn-weighteddistance}
    \de_{\theta}(y)=\sup_{r\geq0}\{r:B_r(y-r\cos\theta E_{n+1})\cap \pr\Omega=\varnothing\}.
\end{align}
One sees from definition that \begin{align}\label{eq-dist-theta}
    \de_\theta(y)=\de(y-\de_\theta(y)\cos\theta E_{n+1})
\end{align}
and for any $0<r<\de_{\theta}(y)$, there holds
\begin{align}\label{ineq-weighted-distance}
   \de(y-r\cos\theta E_{n+1})>r.
\end{align}
See \cref{figure2}.	\begin{figure}[H]
	\centering
	\includegraphics[height=7cm,width=12cm]{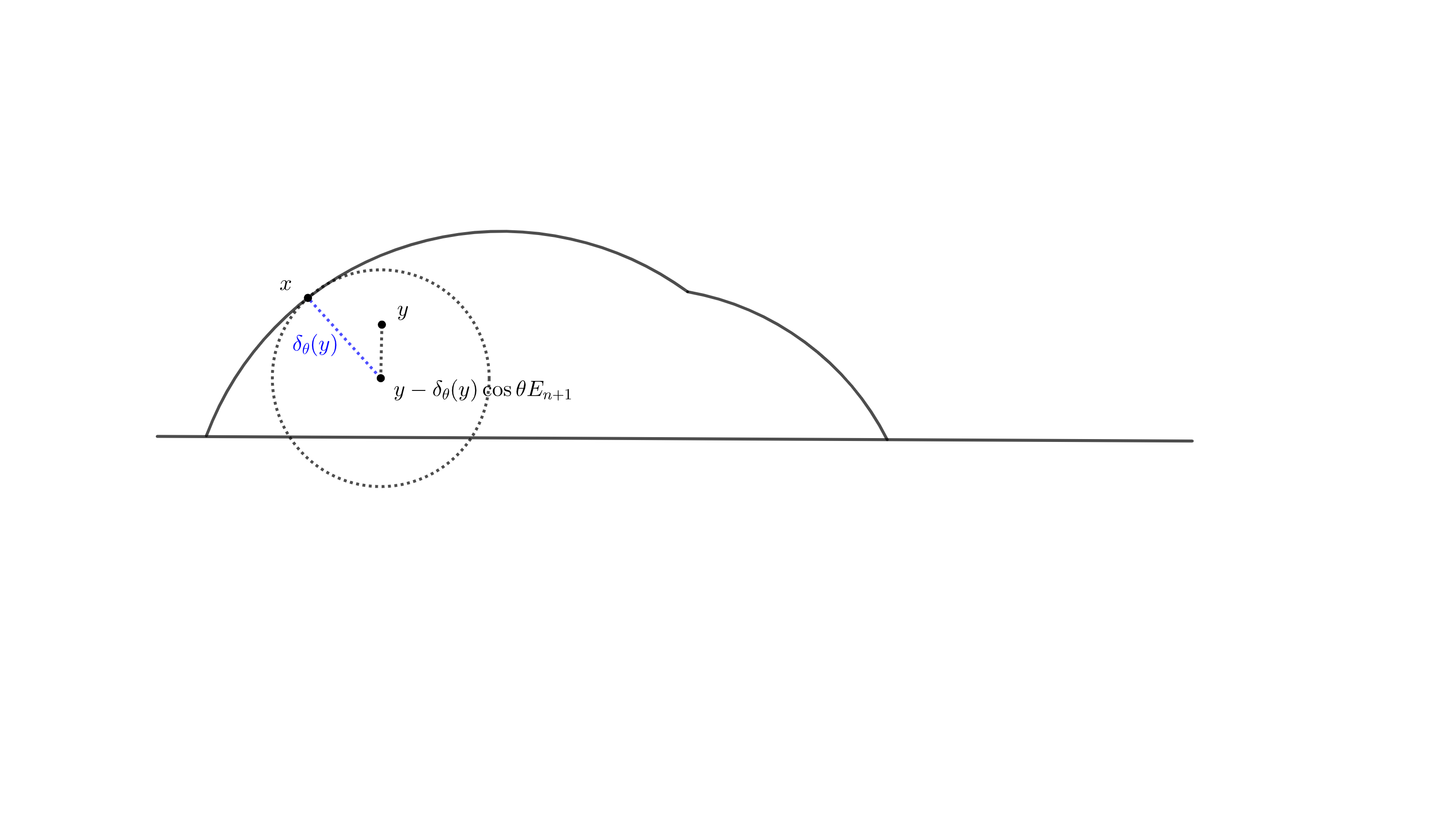}
	\caption{shifted distance function}
	\label{figure2}
\end{figure}
For $s>0$, we define the super level-set and level-set of $\de_{\theta}$ in $\overline\Omega$ by
\begin{align}\label{defn-Omegas}
    \Omega_s:=\{y\in\overline\Omega:\de_{\theta}(y)>s\},
    \quad\pr\Omega_s:=\{y\in\overline\Omega:\de_{\theta}(y)=s\}.
\end{align}
$\delta_\theta$ indeed plays the same role as the distance function considered in \cite{DM19}, and we may
adapt Delgadino-Maggi's approach to define the large subset $\Om^\star_\theta$ of good points in our setting as follows.
\begin{definition}[$\Gamma_{s;\theta}^t$ and $\Gamma_{s;\theta}^+$]\label{defn-Gamma_s^t}
\normalfont
Let $\Om$ be a bounded, relatively open set of finite perimeter in $\overline{\mfR^{n+1}_+}$ and $\theta\in(0,\pi)$.
For every $0<s<t<\infty$, we define $\Gamma_{s;\theta}^t$ to be the set of  $y\in\pr\Omega_s$ such that there exists a geodesic $\gamma_y: [0, t]\to\mfR^{n+1}$ with $\gamma_y(s)=y$ and $\de_\theta(\gamma_y(r))=r$ for every $r\in[0,t]$.

 Moreover, for every $0<s<\infty$, define
 \begin{align*}
\Gamma_{s;\theta}^+=\bigcup_{t>0}\Gamma_{s;\theta}^t,\quad \Om^\star_\theta=\bigcup_{s>0}\Gamma_{s;\theta}^+.
\end{align*}
\end{definition}

In the first step, we shall prove that $\Gamma_{s;\theta}^t$ is $C^{1,1}$-rectifiable (\cref{Thm-C11Rec}). 
Our strategy for proving the $C^{1,1}$-rectifiability theorem in terms of the shifted distance function follows largely from \cite[Theorem 1, step 1]{DM19}. With the $C^{1,1}$-rectifiability theorem, we are able to prove that $|\Om\setminus  \Om^\star_\theta|=0$.
Next, we prove $|\Om^\star_\theta\setminus  \zeta_\theta(Z)|=0$, following \cite[Theorem 1, step 4]{DM19}.
In contrast to the closed hypersurface case in \cite[Theorem 1, step 4]{DM19}, here we have to carefully deal with the boundary of the singular capillary CMC hypersurface $\pr\Om$ (see \cref{Defn-CapillaryCMC} for the definition).
Indeed, a crucial issue to be clarified is the counting measure of $g_{s;\theta}^{-1}(x)$ for $x\in{\rm reg}\pr\Om$ (see the proof of \cref{Thm-Alexandrov} for the definition of $g_{s;\theta}^{-1}$), and when $x$ lies in the interior of $\pr\Om$, i.e., $x\in\pr\Om\setminus\Gamma$, we could use the same argument in \cite{DM19}, by studying the blow-up feature of $\pr\Om$ at $x$, and the maximum principle for stationary rectifiable cones, to conclude that the counting measure is at most $2$.
Likewise,
we have to study the boundary behaviors of $\pr\Om$ and $T$ along $\Gamma$, and consider the blow-up process at every $x\in\Gamma$.
To this end, we first introduce a new definition of $\theta$-stationary triple in \cref{Defn-contactangle-triple} (which is satisfied by the blow-up limit of the singular capillary CMC hypersurface under our regularity assumption in \cref{Thm-Alexandrov}), and then we exploit new boundary strong maximum principles \cref{Lem-MP3}, \cref{Cor-MP3} to study the blow-up process.

Once the properties $|\Om\setminus  \Om^\star_\theta|=0$ and $|\Om^\star_\theta\setminus  \zeta_\theta(Z)|=0$ are established, we can proceed Jia-Wang-Xia-Zhang's argument  \cite{JWXZ22-B} in the framework of sets of finite perimeter.

In the end, we make further comments on the technical assumption $\mcH^{n-1}({\rm sing}\pr\Om)=0$ in \cref{Thm-Alexandrov}.
\begin{remark}\label{Rem-natural-assumption}
    \normalfont
    As already discussed, in \cref{Thm-Alexandrov} $\mcH^{n-1}({\rm sing}\pr\Om)=0$ is a technical but crucial assumption.
    Here we make some illustrations.
    \begin{enumerate}
        \item To study the boundary behavior of $\pr\Om$ and $T$ along $\Gamma$, we need to consider the $\theta$-stationary triple, which is well-defined thanks to \cref{Lem-divergenctheorem}.
        The proof of \cref{Lem-divergenctheorem} is built on the smooth cut-off functions near singularities of $\pr\Om$, whose existence relies largely on the assumption that $\mcH^{n-1}({\rm sing}\pr\Om)=0$.
        \item As we have mentioned in \cref{Rem-assumption-mainThm}(1), Young's law can be deduced from $\mcH^n({\rm sing}\pr\Om)=\mcH^{n-1}({\rm sing}\Gamma)=0$ and the $\mcA$-stationarity of $\Om$ thanks to the Allard's regularity theorem for free boundary rectifiable varifolds \cite[Theorem 4.13]{GJ86}.
        However, it is still an open question whether an Allard-type regularity theorem holds for capillary submanifolds with general contact angles.
        Despite the lack of the Allard-type regularity result, we could still obtain Young's law for $\mcA$-stationary set in \cref{Prop-Young}, provided that $\mcH^{n-1}({\rm sing}\pr\Om)=0$.
    \end{enumerate}
\end{remark}

\noindent{\bf Organization of the paper.}
In \cref{Sec-2} we collect some background material from geometric measure theory and prove the boundary maximum principles that are useful for the blow-up analysis in the proof of the Alexandrov-type theorems.
In \cref{Sec-3} we study the fine properties of $\Gamma_{s;\theta}^t$ and then we prove the rectifiability result \cref{Thm-C11Rec}.
In \cref{Sec-5}, we prove our main result \cref{Thm-Alexandrov}.

\

\noindent{\bf Acknowledgements.}
We are indebted to Professor Guofang Wang for stimulating discussions on this topic and his constant support.

\section{Preliminaries}\label{Sec-2}

\subsection{Notations}\label{Sec-2-1}
When considering the topology of $\mfR^{n+1}$,
we denote by $\overline{\Om}$ the topological closure of a set $\Om$, by ${\rm int}(\Om)$ the topological interior of $\Om$, and by $\p \Om$ the topological boundary of $\Om$.
In terms of the subspace topology (relative topology), we use the following notations:
let $X$ be a topological space and $S$ be a subspace of $X$, we use
${\rm cl}_XS, {\rm int}_XS, \p_XS$
to denote the closure, the interior, and the boundary, respectively, of $S$ in the topological space $X$.
In particular,
for a bounded, relatively open set of finite perimeter $\Om\subset\overline{\mfR^{n+1}_+}$.
We denote by $\pr\Om
:=\p_{\overline{\mfR^{n+1}_+}}\Om
=\overline{\p\Om\cap\mfR^{n+1}_+}$ the relative boundary of $\Om$ in the upper half-space $\overline{\mfR^{n+1}_+}$,
let $T:=
{\rm cl}_{\p\mfR^{n+1}_+}\left(\p\Om\setminus\pr\Om\right)$,
and denote by
$\Gamma:=\pr\Om\cap\p\mfR^{n+1}_+$,
see \cref{figure1} for illustration.
	\begin{figure}[H]
	\centering
	\includegraphics[height=6cm,width=10cm]{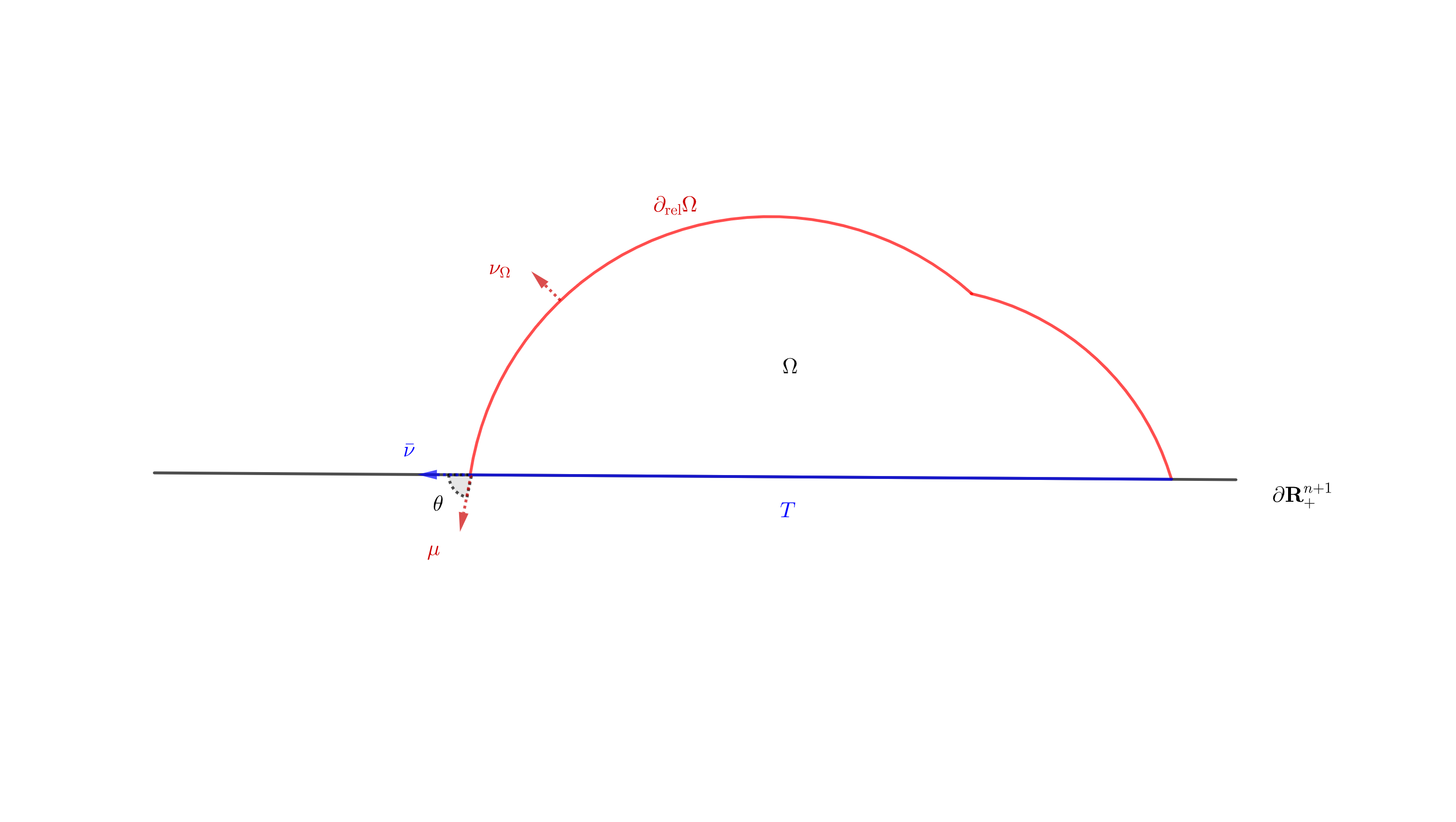}
	\caption{Notations}
	\label{figure1}
\end{figure}
\begin{itemize}
    \item $\bsmu_r$ is the homothety map $y\mapsto ry$;
    \item $\bstau_x$ is the translation map $y\mapsto y-x$;
    \item $\bseta_{x,r}$ is the composition $\bsmu_{r^{-1}}\circ\bstau_x$, i.e., $y\mapsto \frac{y-x}{r}$;
    \item $\mcH^k$ is the $k$-dimensional Hausdorff measure on $\mfR^{n+1}$;
    \item $\mcL^{n+1}$ is the Lebesgue outer measure on $\mfR^{n+1}$;
    \item $B_r(x)$ is the open ball in $\mfR^{n+1}$, centered at $x$ with radius $r$;
    \item $\omega_n$ is the volume of the $n$-dimensional unit ball in $\mfR^{n+1}$;
    \item ${\rm spt}$ is the support of a measure, see \cite[Section 2.4]{Mag12}.
    \item In this paper, we work with the following spaces of vector fields: 
    \begin{align*}
    	&\mathfrak{X}_c(\mfR^{n+1}):=\{\text{the space of all $C^1$-vector fields on }\mfR^{n+1}\text{ with compact support}\},\\
    	&\mathfrak{X}(\overline{\mfR^{n+1}_+}):=\left\{X\in\mathfrak{X}_c(\mfR^{n+1}):X(p)\in T_p\overline{\mfR^{n+1}_+}\text{ for all }p\in\overline{\mfR^{n+1}_+}\right\},\\
    	&\mathfrak{X}_c(\overline{\mfR^{n+1}_+}):=\left\{X\in\mathfrak{X}(\overline{\mfR_+^{n+1}}):X\text{ has relatively compact support in   the open half-space }\mfR^{n+1}_+\right\},\\
    	&\mathfrak{X}_t(\overline{\mfR^{n+1}_+}):=\left\{X\in\mathfrak{X}(\overline{\mfR^{n+1}_+}):X(p)\in T_p(\p\mfR^{n+1}_+)\text{ for all }p\in\p\mfR^{n+1}_+\right\}.
    \end{align*}
    Notice that at any $x\in\p\mfR^{n+1}_+$, $T_x\mfR^{n+1}_+$ is exactly the $(n+1)$-dimensional half-space in $\mfR^{n+1}$ with boundary $T_x\p \mfR^{n+1}_+$.
\end{itemize}


\subsection{Rectifiable sets}\label{Sec-2-3}
A Borel set $M\subset\mfR^{n+1}$ is a \textit{locally $\mcH^n$-rectifiable set} if $M$ can be covered, up to a $\mcH^n$-negligible set, by countably many Lipschitz images of $\mfR^n$ into $\mfR^{n+1}$, and if $\mcH^n\llcorner M$ is locally finite on $\mfR^{n+1}$. $M$ is called \textit{$\mcH^n$-rectifiable} if, in addition, $\mcH^n(M)<\infty$; $M$ is called \textit{normalized}, if $M={\rm spt}(\mcH^n\llcorner M)$, i.e.,
\begin{align*}
	x\in M\quad\text{if and only if }\mcH^n(B_\rho(x)\cap M)>0\quad\text{for all }\rho>0.
\end{align*}
\begin{proposition}[{\textit{area formula for $k$-rectifiable sets}, \cite[Theorem 11.6]{Mag12}}]
	For $1\leq k\leq m$, if $A\subset\mfR^n$ is a locally $\mcH^k$-rectifiable set and $f:\mfR^n\ra\mfR^m$  is a Lipschitz map, then
	\begin{align}\label{formu-area}
		\int_{\mfR^m}\mcH^0\left(A\cap \left\{ f=y\right\} \right)\rd\mcH^k(y)=\int_A {\rm J}^Af(x)\rd\mcH^k(x),
	\end{align} 
	where $\left\{ f:=y\right\}=\left\{x\in\mfR^n: f(x)=y\right\}$, ${\rm J}^Af(x)$ is the Jacobian of $f$ with respect to $A$ at $x$ (see for example \cite[(11.1)]{Mag12}), which exsits for $\mcH^k$-a.e. $x\in A$.
\end{proposition}

\begin{lemma}[{\textit{tangential property of Lipschitz function along rectifiable sets}, \cite[Section 2.1(iv)]{DM19}}]\label{DM-2.1(iv)}
	Let $M\subset\mfR^{n+1}$ be a locally $\mcH^n$-rectifiable set and $f:M\ra\mfR^{n+1}$ is a Lipschitz map defined on $M$, then for any Lipschitz functions $F,G:\mfR^{n+1}\ra\mfR^{n+1}$ such that $F=G=f$ on $M$, we have
	\begin{align}
		\nabla^MF=\nabla^MG\quad\mcH^n\text{-a.e. on }M. 
	\end{align}
	In particular, if $g:\mfR^n\ra\mfR^{n+1}$ is a Lipschitz map and $E\subset\mfR^n$ is a Borel set,
	then $T_xM=(\nabla g)_{g^{-1}(x)}[\mfR^n]$ for $\mcH^n$-a.e. $x\in M\cap g(E)$, with
	\begin{align}\label{eq-DM-2.5}
		(\nabla^MF)_x[\tau]
		=\nabla(F\circ g)_{g^{-1}(x)}[(\nabla g)_x^{-1}[\tau]]\quad\forall\tau\in T_xM.
	\end{align}
	Here $\nabla^MF(x)$ denotes the tangential differential of $F$ with respect to $M$ at $x$, which exsits for $\mcH^n$-a.e. $x\in M$ by virtue of the Rademacher-type theorem \cite[Theroem 11.4]{Mag12}.
\end{lemma}

\subsection{Rectifiable varifolds}\label{Sec-2-4}

We now quickly recall some basic notions of rectifiable varifolds in $\mfR^{n+1}$, and we refer to the standard references \cite{Allard72,Sim83} for details.

Let $M$ be a locally $\mcH^k$-rectifiable set and consider a Borel measurable function $\psi\in L^1_{\rm loc}(\mcH^k\llcorner M;\mfR^+)$.
The rectifiable varifold $V={\rm var}(M,\psi)$ defined by $M$ and $\psi$ is the Radon measure on $\mfR^{n+1}\times G(n+1,k)$, defined as
\begin{align*}
	\int_{\mfR^{n+1}\times G(n+1,k)}\Phi\rd{\rm var}(M,\psi)
	=\int_M\Phi(x,T_xM)\psi(x)\rd\mcH^k(x),
\end{align*}
for every bounded, compactly supported Borel function $\Phi$ on $\mfR^{n+1}\times G(n+1,k)$.
In particular, for any $X\in\mathfrak{X}_c(\mfR^{n+1})$, the well-known first variatioanl formula reads
\begin{align*}
	\de{\rm var}(M,\psi)[X]
	=\int_M{\rm div}_MX(x)\psi(x)\rd\mcH^k(x).
\end{align*}

The \textit{weight measure} of $V={\rm var}(M,\psi)$ is denoted by $\vvert V\vvert=\psi\mcH^k\llcorner M$.
As in \cite[Definition 42.3]{Sim83}, we denote ${\rm VarTan(V,x)}$ to be the set of \textit{varifold tangents} of $V$ at some $x\in{\rm spt}\vvert V\vvert.$
By the compactness of Radon measures \cite[Theorem 4.4]{Sim83}, ${\rm VarTan}(V,x)$ is compact and non-empty provided that the upper density $\Theta^{\ast k}(\vvert V\vvert,x)$ is finite.


\begin{definition}[Stationary varifolds]
	\normalfont
	A rectifiable varifold ${\rm var}(M,\psi)$ is said to be stationary in $\mfR^{n+1}$ if $\delta{\rm var}(M,\psi)[X]=0$ for any $X\in\mathfrak{X}_c(\mfR^{n+1})$,
	and is said to be stationary in $\overline{\mfR^{n+1}_+}$ with free boundary if $\delta {\rm var}(M,\psi)[X]=0$ for any $X\in\mathfrak{X}_t(\overline{\mfR^{n+1}_+})$.
\end{definition}
The following useful reflection trick will be needed in our proof, see \cite[(3.2)]{Allard75}, \cite[Remark 4.11(iii)]{GJ86} and \cite[Lemma 2.2]{LZ21-a}.
\begin{lemma}[Reflection Principle]\label{Lem-ReflectionPrinciple}
	Let $E_{n+1}:=(0,\ldots,0,1)\in\overline{\mfR^{n+1}_+}$.
	Let $\theta_{n+1}:\mfR^{n+1}\ra\mfR^{n+1}$ denote the reflection map about the unit vector $E_{n+1}$, i.e., $\theta_{n+1}(u)=u-2(u\cdot E_{n+1})E_{n+1}$.
	For any rectifiable varifold $V:={\rm var}(M,\psi)$, define the doubled varifold
	\begin{align}\label{defn-barV}
		\bar V:=V+(\theta_{n+1})_\#V.
	\end{align}
	If $V$ is stationary in $\overline{\mfR^{n+1}_+}$ with free boundary,
	then $\bar V$ is stationary in $\mfR^{n+1}$.
\end{lemma}

\subsection{Sets of finite perimeter}\label{Sec-2-2}
For basic knowledge regarding sets of finite perimeter, we refer to the monograph \cite{Mag12} (in particular, Chapter 15) for a detailed account.

Given a Lebesgue measurable set $\Om\subset\mfR^{n+1}$, we say that $\Om$ is a \textit{set of finite perimeter in $\mfR^{n+1}$} if \begin{align*}
	\sup\left\{\int_{\mfR^{n+1}} {\rm div}X\rd\mcL^{n+1}: X\in C^1_c(\mfR^{n+1};\mfR^{n+1}),  \vert X\vert\leq1\right\}<\infty.
\end{align*}
An equivalent characterization of sets of finite perimeter (see \cite[Proposition 12.1]{Mag12}) is that: there exists a $\mfR^{n+1}$-valued Radon measure $\mu_\Om$ on $\mfR^{n+1}$ such that for any $X\in C_c^1(\mfR^{n+1};\mfR^{n+1})$,
\begin{align}\label{eq-DM19-2-10}
	\int_\Om{\rm div}X\rd\mcL^{n+1}
	=\int_{\mfR^{n+1}}X\cdot\rd\mu_\Om.
\end{align}
$\mu_\Om$ is called the \textit{Gauss-Green measure of $\Om$}.
The \textit{relative perimeter of $\Om$ in $E\subset\mfR^{n+1}$}, and the \textit{perimeter of $\Om$}, are defined as
\begin{align*}
	P(\Om;E)=\vert\mu_\Om\vert(E),\quad P(\Om)=\vert\mu_\Om\vert(\mfR^{n+1}).
\end{align*}
Regarding the topological boundary of a set of finite perimeter $E$, one has (see \cite[Proposition 12.19]{Mag12})
\begin{align*}
	{\rm spt}\mu_\Om
	=\{x\in\mfR^{n+1}:0<\vert \Om\cap B_r(x)\vert<\om_{n+1}r^{n+1},\quad\forall r>0\}\subset\p\Om.
\end{align*}

The \textit{reduced boundary} $\p^\ast \Om$ is the set of those $x\in{\rm spt}\mu_\Om$ such that the limit
\begin{align*}
	\lim_{r\ra0^+}\frac{\mu_\Om(B_r(x))}{\vert\mu_\Om\vert(B_r(x))}\text{ exists and belongs to }\mfS^{n}.
\end{align*}
The Borel vector field $\nu_\Om:\p^\ast\Om\ra\mfS^n$ is called the \textit{measure-theoretic outer unit normal to $\Om$}, and there holds
\begin{align*}
	\overline{\p^\ast\Om}={\rm spt}\mu_\Om.
\end{align*}
Moreover, when $\Om$ is a bounded, relatively open set of finite perimeter in $\overline{\mfR^{n+1}_+}$, the reduced boundary is locally $\mcH^n$-rectifiable and the sets $\pr\Om,T$ are normalized.
By \cite[(2.1)]{DePM15},
\begin{align*}
	\mu_\Om=\nu_\Om\mcH^n\llcorner(\p^\ast\Om\cap \mfR^{n+1}_+)-E_{n+1}\mcH^n\llcorner(\p^\ast\Om\cap\p\mfR^{n+1}_+),
\end{align*}
and thus \eqref{eq-DM19-2-10} reads
\begin{align*}
	\int_\Om{\rm div}X\rd\mcL^{n+1}
	=\int_{\p^\ast\Om\cap\mfR^{n+1}_+}X\cdot\nu_\Om\rd\mcH^n
	-\int_{\p^\ast\Om\cap\p\mfR^{n+1}_+}X\cdot E_{n+1}\rd\mcH^n.
\end{align*}
Under the regularity assumption $\mcH^{n}({\rm sing}\pr\Om)=0$, the above equality then takes the form
\begin{align}\label{formu-GaussGreen}
	\int_\Om{\rm div}X\rd\mcL^{n+1}
	=\int_{\pr\Om}X\cdot\nu_\Om\rd\mcH^n-\int_T X\cdot E_{n+1}\rd\mcH^n.
\end{align}

To every set of finite perimeter in $\overline{\mfR^{n+1}_+}$, we can always associate in a natural way the integer rectifiable varifolds ${\rm var}(\p^\ast\Om\cap\mfR^{n+1}_+,1)$ and ${\rm var}(\p^\ast\Om\cap\p\mfR^{n+1}_+,1)$. For simplicity we denote them respectively by ${\rm var}(\p^\ast\Om\cap\mfR^{n+1}_+)$ and ${\rm var}(\p^\ast\Om\cap\p\mfR^{n+1}_+)$.

\subsection{Critical points of the $\mcA$-functional}\label{Sec-2-6}

Recall that the $\mcA$-functional
is defined in \cref{Defn-A-functional}, which follows from \cite[Definition 1.1]{LZZ21}.
In fact, given $\theta\in(0,\pi)$, a constant $c>0$, and a set of finite perimeter $\Om$ as in \cref{Defn-A-functional},
the first variation for $\mcA$ along $X\in\mathfrak{X}(\overline{\mfR^{n+1}_+})$ is (see \cite[(1.7)]{LZZ21})
\begin{align}\label{formu-LZZ21-1.7}
	\de \mcA\mid_\Om(X)
	=\int_{\p^\ast\Om\llcorner\mfR^{n+1}_+}{\rm div}_{\p\Om}X\rd\mcH^n+\cos(\pi-\theta)\int_{\p^\ast\Om\llcorner\p\mfR^{n+1}_+}{\rm div}_{\p\Om}X\rd\mcH^n-c\int_{\p^\ast\Om}X\cdot\nu_{\Om}\rd\mcH^n,
\end{align}
which implies the following facts:
if $\Om$ is stationary for $\mcA$ as in \cref{Defn-A-functional}, then
\begin{enumerate}
	\item when $\theta\in(0,\pi)$, $\Om$ is stationary for the free energy functional
	\begin{align*}
		\mathcal{E}(\Om)=P(\Om; \mfR^{n+1}_+)-\cos\theta P(\Om; \p\mfR^{n+1}_+).
	\end{align*}
	under volume constraint, see e.g., \cite[Definition 4.2]{XZ21}, and the reason is that the first variation of $\vert\Om\vert$ along $X$ is indeed given by $\int_{\p^\ast\Om}X\cdot\nu_{\Om}\rd\mcH^n$,
	see \cite[(17.31)]{Mag12}.
	\item when $\theta\in[\pi/2,\pi)$, the naturally induced varifold ${\rm var}(\p^\ast\Om\cap\mfR^{n+1}_+)$ has a fixed contact angle $\theta$ with $\p\mfR^{n+1}_+$ at $\p^\ast\Om\cap\p\mfR^{n+1}_+$ in the sense of \cite[Definition 3.1]{KT17} with generalized mean curvature $H=c$, since for any $X\in\mathfrak{X}_t(\overline{\mfR^{n+1}_+})$,
	\eqref{formu-LZZ21-1.7} reads as
	\begin{align}\label{eq-first-variation}
		\int_{\p^\ast\Om\llcorner\mfR^{n+1}_+}{\rm div}_{\p\Om}X\rd\mcH^n+\cos(\pi-\theta)\int_{\p^\ast\Om\llcorner\p\mfR^{n+1}_+}{\rm div}_{\p\mfR^{n+1}_+}X\rd\mcH^n
		=\int_{\p^\ast\Om\llcorner\mfR^{n+1}_+}X\cdot (c\nu_{\Om})\rd\mcH^n.
	\end{align}
	
	Moreover, the naturally induced pair of varifolds $\left({\rm var}(\p^\ast\Om\cap\mfR^{n+1}_+),{\rm var}(\p^\ast\Om\cap \p\mfR^{n+1}_+)\right)$ satisfies the contact angle condition $\theta$ in the sense of \cite[Definition 3.1]{DeMDeP21}.

    For the case $\theta\in(0,\pi/2)$, we have the following observation.
\begin{remark}\label{Rem-replacement}
	\normalfont
	When the capillary angle $\theta\in(0,\pi/2)$, since we consider only the sets of finite perimeter that are bounded, we may assume that there is a large enough smooth open set $\mathbf{B}_+$ such that $\Om\subset \mathbf{B}_+\subset\overline{\mfR^{n+1}_+}$.
	Once the set $\mathbf{B}_+$ is fixed, we know that 
    $P(\mathbf{B}_+;\p\mfR^{n+1}_+)$ is a fixed number and it is clear that 
    $$P(\Om;\p\mfR^{n+1}_+)=P(\mathbf{B}_+;\p\mfR^{n+1}_+)-P(\mathbf{B}_+\setminus\Om;\p\mfR^{n+1}_+),$$
	which implies the following fact:
	if $\Om$ is $\mcA$-stationary (as in \cref{Defn-A-functional}), then
	it is stationary for the modified functional
    \begin{align}\label{defn-tilde-A}
        \tilde\mcA(\Om):=P(\Om;\mfR^{n+1}_+)+\cos\theta P(\mathbf{B}_+\setminus\Om;\p\mfR^{n+1}_+)-c\vert\Om\vert,
    \end{align}
	in the sense that for any $C^1$-diffeomorphism $f:\overline{\mfR^{n+1}_+}\ra\overline{\mfR^{n+1}_+}$ with ${\rm spt}f\subset\subset\overline{\mathbf{B}_+}$,
    such that $f:\p\mfR^{n+1}_+\ra\p\mfR^{n+1}_+$ is a diffeomorphism of $\p\mfR^{n+1}_+$,
	there holds
	\begin{align*}
		\frac{\rd}{\rd t}\mid_{t=0}\tilde\mcA(\psi_t(\Om))=0,
	\end{align*}
	where $\{\psi_t\}_{\vert t\vert<T}$ is a one parameter family of diffeomorphisms induced by $f$.
	
\end{remark}
\end{enumerate}

An important feature we shall use for the $\mcA$-stationary sets is that they satisfy the following Euclidean volume growth condition:

\begin{proposition}\label{Prop-Evg}
	Given $\theta\in(0,\pi)$ and $c>0$.
	Let $\Om\subset\overline{\mfR^{n+1}_+}$ be a non-empty, bounded, relatively open set with finite perimeter in $\overline{\mfR^{n+1}_+}$ that is stationary for the $\mcA$-functional, then $\pr\Om$ has Euclidean volume growth, that is, there exists some universal constants (depends only on $\Om$ and $n$) $R_1>0$ and $C_1>0$ such that for any $x\in \pr\Om$ and for any $0<r<R_1$, there holds
	\begin{align}\label{ineq-vg-M}
	\mcH^n(\pr\Om\cap B_r(x))\leq C_1r^n.
		\end{align}
\end{proposition}
\begin{proof}
{\bf Case 1. }$\theta\in[\pi/2,\pi)$.

	Our starting point is that as illustrated below \eqref{formu-LZZ21-1.7},
	${\rm var}(\p^\ast\Om\llcorner\mfR^{n+1}_+)$ has a fixed contact angle $\theta$ with $\p\mfR^{n+1}_+$ at $\p^\ast\Om\llcorner\p\mfR^{n+1}_+$ in the sense of \cite[Definition 3.1]{KT17} with bounded generalized mean curvature,
	therefore the monotonicity formula \cite[Theorem 3.2]{KT17} is applicable here. Moreover, since $\p\mfR^{n+1}_+$ is planar, we know that the maximal distance ($s_0$) used to define the tubular neighborhood ($N_{s_0}$) in \cite{KT17} can be taken as large as possible in our case.
	By taking $s_0=6\cdot{\rm diam}(\Om):=6\cdot\max_{x,y\in\overline\Om}\vert x-y\vert$ and fix $p=n+1$ in \cite[(3.9)]{KT17}, we find: for any $x\in\pr\Om$ and for any $0<r<{\rm diam}(\Om)$,
	\begin{align*}
		\frac{\mcH^{n}(\pr\Om\cap B_r(x))}{r^{n}}
		\leq\om_{n}\left\{\left(\frac{2\mcH^{n}(\p^\ast\Om)}{\om_{n}({\rm diam}(\Om))^{n}}\right)^{\frac{1}{n+1}}+\gamma\cdot({\rm diam}(\Om))^{\frac{1}{n+1}}\right\}^{n+1}:=C,
	\end{align*}
	where $\gamma:=
	\left(\frac{1}{\om_{n}}\int_{\p^\ast\Om\cap\mfR^{n+1}_+}2\vert H(x)\vert^{n+1}\rd\mcH^n(x)\right)^{\frac{1}{n+1}}$ is bounded since $\Om$ is of finite perimeter and $\vert H\vert=c$ (which follows from the $\mcA$-stationarity of $\Om$),
	and hence
	\begin{align*}
		\mcH^{n}(\pr\Om\cap B_r(x))\leq Cr^{n}
	\end{align*}
	for every $x\in\pr\Om$, $r<{\rm diam}(\Om)$.
	This shows that $\pr\Om$ has Euclidean volume growth.

 {\bf Case 2. }$\theta\in(0,\pi/2)$.

Recall \cref{Rem-replacement},
${\rm var}(\p^\ast\Om\cap\mfR^{n+1}_+)$ has a fixed contact angle $\tilde\theta:=\pi-\theta\in(\pi/2,\pi)$ with $\p\mfR^{n+1}_+$ at $\p^\ast(\mathbf{B}_+\setminus\Omega)\cap\p\mfR^{n+1}_+$ in the sense of \cite[Definition 3.1]{KT17} with generalized mean curvature $H=c$,
a simple modification of {\bf Case 1} then shows that:
when $\theta\in(0,\pi/2)$, $\pr\Om$ satisfies the Euclidean volume growth condition as well.
 This completes the proof.
\end{proof}
\begin{remark}\label{Rem-integrals}
	\normalfont
	Our main theorem is stated under the regularity assumption that $\mcH^{n-1}({\rm sing}\pr\Om)=0$.
	In light of the Hausdorff dimension of the singular set, in all follows, we do not make a distinction between the integrals $\int_{\pr\Om}\cdot \rd\mcH^n$, $\int_{\p^\ast\Om\cap\mfR^{n+1}_+}\cdot\rd\mcH^n$ and $\int_{{\rm reg}\pr\Om}\cdot \rd\mcH^n$;
	$\int_\Gamma\cdot \rd\mcH^{n-1}$ and $\int_{{\rm reg}\Gamma}\cdot\rd\mcH^{n-1}$.
\end{remark}
Under the regularity assumption, we have the following useful tangential divergence theorem for the $\mcA$-stationary set, whose proof is postponed to \cref{App-1}.
\begin{lemma}\label{Lem-divergenctheorem}
	Given $\theta\in(0,\pi)$ and $c>0$, let $\Om$ be a non-empty, bounded, relatively open set with finite perimeter in $\overline{\mfR^{n+1}_+}$ such that $\mcH^{n-1}({\rm sing}\pr\Om)=0$.
	If $\Om$ is stationary for the $\mcA$-functional,
	then for any $X\in \mathfrak{X}(\overline{\mfR^{n+1}_+})$, there holds
	\begin{align}\label{formu-div-M}
		\int_{\pr\Om}{\rm div}_{\p\Om}X\rd\mcH^{n}=
		\int_\Gamma X\cdot\mu\rd\mcH^{n-1}+\int_{\pr\Om} X\cdot (c\nu_\Om)\rd\mcH^{n},
	\end{align}
	and for any $X\in\mathfrak{X}_t(\overline{\mfR^{n+1}_+})$, there holds
	\begin{align}\label{formu-div-B+}
		\int_{T}{\rm div}_{\p\mfR^{n+1}_+}X\rd\mcH^{n}
		=\int_\Gamma X\cdot\bar\nu\rd\mcH^{n-1}.
	\end{align}
\end{lemma}

\subsection{Singular capillary CMC hypersurfaces in a half-space}\label{Sec-2-8}
Let us first collect the basic facts on the $\mcA$-stationary sets and then give a formal definition of the singular capillary CMC hypersurfaces in a half-space.
Given a fixed capillary contact angle $\theta\in(0,\pi)$ and a positive constant $c$.
Let $\Om$ be a non-empty, bounded, relatively open set of finite perimeter in $\overline{\mfR^{n+1}_+}$ that is stationary for the $\mcA$-functional.
We learn from \cref{Prop-Evg} that $\pr\Om$ satisfies the Euclidean volume growth condition \eqref{ineq-vg-M}.
On the other hand, provided that $\mcH^{n-1}({\rm sing}\pr\Om)=0$, we can use the tangential divergence theorems on $\pr\Om$ and $T$ as in \cref{Lem-divergenctheorem}.

Having these facts in mind and recall that the $\mcA$-stationary set $\Om$ is indeed stationary for the free energy functional $\mathcal{E}(\Om)$ under volume constraint (see below \eqref{formu-LZZ21-1.7}), we thus arrive at
\begin{proposition}[{\cite[Proposition 4.3]{XZ21}}]\label{Prop-Young}
	Given $\theta\in(0,\pi)$ and $c>0$, let $\Om$ be a non-empty, bounded, relatively open set with finite perimeter in $\overline{\mfR^{n+1}_+}$ such that $\mcH^{n-1}({\rm sing}\pr\Om)=0$.
	If $\Om$ is stationary for the $\mcA$-functional, then Young's law holds.
	Precisely,
	on ${\rm reg}\Gamma={\rm reg}\pr\Om\cap\p\mfR^{n+1}_+$,
	the measure-theoretic hypersurface $\pr\Om$ meets $\p\mfR^{n+1}_+$ with a constant contact angle $\theta$, i.e.,
	\begin{align}\label{eq-Young}
		\nu_\Om\cdot(-E_{n+1})
		=-\cos\theta
		=-\mu\cdot\bar\nu\quad\text{on }{\rm reg}\Gamma.
	\end{align}
\end{proposition}
This shows that the fixed capillary angle $\theta$ used to define the $\mcA$-functional in \cref{Defn-A-functional} is indeed the contact angle of ${\rm reg}\pr\Om$ with $\p\mfR^{n+1}_+$, and hence the following definition makes sense.
\begin{definition}\label{Defn-CapillaryCMC}
	\normalfont
	Given $\theta\in(0,\pi)$ and $c>0$, let $\Om$ be a non-empty, bounded, relatively open set with finite perimeter in $\overline{\mfR^{n+1}_+}$ such that $\mcH^{n-1}({\rm sing}\pr\Om)=0$.
	If $\Om$ is $\mcA$-stationary,
	then we say that $\pr\Om$ is a \textit{singular capillary CMC hypersurface in $\overline{\mfR^{n+1}_+}$}.
\end{definition}
It is clear that \cref{Defn-CapillaryCMC} holds true for the $\Om$ we consider in \cref{Thm-Alexandrov}, and
an important fact on the singular capillary CMC hypersurface we shall use is the following Minkowski-type formula, see \cite{AS16} in the smooth setting.
\begin{proposition}[Minkowski-type formula in the half-space]\label{Prop-Minko}
	Given $\theta\in(0,\pi)$ and $c>0$, let $\pr\Om$ be a singular capillary CMC hypersurface in $\overline{\mfR^{n+1}_+}$.
	There holds
	\begin{align}\label{formu-Minko}
		\int_{\pr\Om}n(1-\cos\theta\nu_\Om\cdot E_{n+1})-x\cdot(c\nu_\Om)\rd\mcH^n=0.
	\end{align}
\end{proposition}
\begin{proof}
	Integrating ${\rm div}(E_{n+1})$ on $\Om$, the generalized Gauss-Green's formula \eqref{formu-GaussGreen} and \cref{Rem-integrals} yields
	\begin{align*}
		0=\int_\Om{\rm div}(E_{n+1})\rd\mcL^{n+1}
		=\int_{\pr\Om}\nu\cdot E_{n+1}\rd\mcH^n-\vert T\vert.
	\end{align*}
	
	On the other hand, since $\Om$ is $\mcA$-stationary, testing \eqref{formu-LZZ21-1.7} with a vector field $X\in\mathfrak{X}_t(\overline{\mfR^{n+1}_+})$ which is the position vector field $X(x)=x$ in a neighborhood of $\Om$, we obtain
	\begin{align*}
		\int_{\pr\Om}n\rd\mcH^n-n\cos\theta\vert T\vert-\int_{\pr\Om}x\cdot(c\nu_\Om)\rd\mcH^n=0.
	\end{align*}
	\eqref{formu-Minko} follows by combining these equalities.
\end{proof}
The Minkowski-type formula results in the following characterization of the given constant $c>0$.
\begin{corollary}
	Given $\theta\in(0,\pi)$ and $c>0$, let $\pr\Om$ be a singular capillary CMC hypersurface in $\overline{\mfR^{n+1}_+}$.
	The constant mean curvature $c$ satisfies
	\begin{align}\label{defn-H0}
		c=H^0_{\Om;\theta}:=\frac{n\int_{\pr\Om}\left(1-\cos\theta\nu_{\Om}\cdot E_{n+1}\right)d\mcH^n}{(n+1)\vert\Om\vert}>0.
	\end{align}
\end{corollary}
\begin{proof}
	Let us consider the position vector field $X(x)=x$, integrating ${\rm div}X$ on the set of finite perimeter $\Om$, using the generalized Gauss-Green formula \eqref{formu-GaussGreen} and \cref{Rem-integrals}, we get
	\begin{align*}
		(n+1)\vert\Om\vert
		=\int_\Om{\rm div}X(x)\rd\mcL^{n+1}(x)
		=&\int_{{\rm reg}\pr\Om}x\cdot\nu_{\Om}(x)\rd\mcH^n(x).
	\end{align*}
	Since $c$ is a constant, we can exploit the Minkowski-type formula \eqref{formu-Minko} to find
	\begin{align*}
		(n+1)c\vert\Om\vert
		=n\int_{\pr\Om}\left(1-\cos\theta\nu_{\Om}\cdot E_{n+1}\right)\rd\mcH^n.
	\end{align*}
	Rearrange the equality and we get \eqref{defn-H0}.
\end{proof}

\subsection{Triple of varifolds that has contact angle $\theta$ in the half-space}\label{Sec-2-9}
As mentioned in the introduction, to study the blow-up process along the boundary, we introduce the following contact angle condition of triple of varifolds.

Given $\theta\in(0,\pi/2)\cup(\pi/2,\pi)$ and $c>0$.
Let $\Om\subset \overline{\mfR^{n+1}_+}$ be a bounded, relatively open set of finite perimeter and finite volume, which is stationary for the $\mcA$-functional with $\mcH^{n-1}({\rm sing}\pr\Om)=0$.
By virtue of Young's law and \cref{Lem-divergenctheorem}, we carry out the classical computation as follows:
for any $X\in\mathfrak{X}(\overline{\mfR^{n+1}_+})$,
\begin{align}\label{eq-divergence-1}
	&\int_{\pr\Om}{\rm div}_{\p\Om}X(x)\rd\mcH^n(x)
	=\int_\Gamma X\cdot\mu \rd\mcH^{n-1}+\int_{\pr\Om}X\cdot (c\nu_\Om) \rd\mcH^n\notag\\
	=&\cos\theta\int_\Gamma X\cdot\bar\nu \rd\mcH^{n-1}-\sin\theta\int_\Gamma X\cdot E_{n+1}\rd\mcH^{n-1}+\int_{\pr\Om} X\cdot(c\nu_\Om)\rd\mcH^n,
\end{align}
here we have used the fact that $\mu=\cos\theta\bar\nu-\sin\theta E_{n+1}$ along ${\rm reg}\Gamma$.
Using \eqref{formu-div-B+} and notice that $T$ is planar, we get
\begin{align}\label{eq-divergence-2}
	\int_{T}{\rm div}_{\p\mfR^{n+1}_+}X\rd\mcH^n
	=\int_{T}{\rm div}_{\p\mfR^{n+1}_+}(X^T+X^\perp)\rd\mcH^n
	=\int_\Gamma X\cdot\bar\nu \rd\mcH^{n-1},
\end{align}
which yields
\begin{align}\label{eq-divergence-3}
	&\int_{\pr\Om}{\rm div}_{\p\Om}X\rd\mcH^n+\cos(\pi-\theta)\int_{T}{\rm div}_{\p\mfR^{n+1}_+}X\rd\mcH^n\notag\\
	=&\int_{\pr\Om}X\cdot(c\nu_\Om)\rd\mcH^n-\sin\theta\int_\Gamma X\cdot(E_{n+1})\rd\mcH^{n-1}.
\end{align}

Enlightened by \cite[Definition 3.1, Proposition 3.1]{DeMDeP21} and the above classical computation, we introduce the following contact angle condition for triple of rectifiable varifolds, which is stronger than \cite[Definition 3.1]{DeMDeP21} since it contains not only the tangential information but also the normal one.
\begin{definition}\label{Defn-contactangle-triple}
	\normalfont
	Given $\theta\in(0,\pi/2)\cup(\pi/2,\pi)$.
	Let $M\subset\overline{\mfR^{n+1}_+}, T\subset\p\mfR^{n+1}_+$ be normalized locally $\mcH^n$-rectifiable sets,
	let $\Gamma\subset\p\mfR^{n+1}_+$ be a normalized locally $\mcH^{n-1}$-rectifiable set, and let $\psi$ be a positive locally $\mcH^n$-integrable function on $M$.
	We say that the triple $({\rm var}(M,\psi),{\rm var}(T),{\rm var}(\Gamma))$ satisfies the contact angle condition $\theta$ if there exists a $\mcH^n\llcorner M$-measurable, $\mcH^n\llcorner M$-integrable vector field $\mfH$
	such that:
	\begin{enumerate}
		\item
		for any $X\in\mathfrak{X}(\overline{\mfR^{n+1}_+})$,
		there holds
		\begin{align}\label{formu-1stvariation-1}
			&\int_{M}\psi{\rm div}_MX\rd\mcH^n+
			\cos(\pi-\theta)\int_{T}{\rm div}_{\p\mfR^{n+1}_+}X\rd\mcH^n\notag\\
			=&-\int_{M}\psi X\cdot \mfH \rd\mcH^n
			-\sin\theta\int_{\Gamma}X\cdot E_{n+1} \rd\mcH^{n-1},
		\end{align}
		\item
		there exists $\bar\nu\in L^1(\p\mfR^{n+1}_+,\mcH^n\llcorner\Gamma)$ such that $\vert\bar\nu\vert=1$ for a.e. $x\in\Gamma$ and satisfies:
		\begin{align}\label{formu-1stvariation-2}
			\int_{T}{\rm div}_{\p\mfR^{n+1}_+}X\rd\mcH^n
			=\int_{\Gamma}X\cdot\bar\nu \rd\mcH^{n-1},\quad\forall X\in\mathfrak{X}_t(\overline{\mfR^{n+1}_+}).
		\end{align}

	\end{enumerate}
	
	In particular, we say that $({\rm var}(M,\psi),{\rm var}(T),{\rm var}(\Gamma))$ is a $\theta$-stationary triple if $\mfH=0$ for a.e. $x\in M$.
	In this case, the first variation formula simply reads
	\begin{align}\label{formu-1stvariation-3}
		&\int_{M}\psi{\rm div}_MX\rd \mcH^n+
		\cos(\pi-\theta)\int_{T}{\rm div}_{\p\mfR^{n+1}_+}X\rd\mcH^n\notag\\
		=&-\sin\theta\int_{\Gamma}X\cdot E_{n+1} \rd\mcH^{n-1},\quad\forall X\in\mathfrak{X}(\overline{\mfR_+^{n+1}}).
	\end{align}
\end{definition}

Note that we have already proved that the triple of varifolds $\left({\rm var}(\pr\Om),{\rm var}(T),{\rm var}(\Gamma)\right)$ satisfies the contact angle condition $\theta$ by virtue of \eqref{eq-divergence-3} and \eqref{eq-divergence-2}.
Now we focus on the blow-up behavior along $\Gamma$ of the triple, and we aim at showing that the blow-up limit of the triple is $\theta$-stationary.
Recall the assumption that $\Gamma$ is a smooth $(n-1)$-manifold in $\p\mfR^{n+1}$ in the Alexandrov-type theorem \cref{Thm-Alexandrov}. A direct consequence of the assumption is that $T$ is now a compact domain with smooth boundary $\Gamma$ in $\p\mfR^{n+1}_+\cong\mfR^n$,
and it is clear that at any $x\in\Gamma$,
any blow-up of $T$ would be a half $n$-plane in $\p\mfR^{n+1}_+$, whose boundary (a $(n-1)$-plane) is exactly the blow-up limit of $\Gamma$. 

Since (see below \eqref{formu-LZZ21-1.7}) the pair of varifolds $\left({\rm var(\pr\Om),{\rm var}(T)}\right)$ has fixed contact angle $\theta$ as in \cite[Definition 3.1]{DeMDeP21} with constant generalized mean curvature, it follows that
when $\theta\in(\pi/2,\pi)$,
the varifold 
\begin{align}\label{defn-Varifold-V}
	V:={\rm var}(\pr\Om)+\cos(\pi-\theta){\rm var}(T)
\end{align}
is a free boundary varifold in $\overline{\mfR^{n+1}_+}$ with constant generalized mean curvature.
Using \cite[Theorem 1.4]{DeMasi21}, we see that the density $\Theta^n(\vvert V\vvert,x)$ exists and is finite for every $x\in\Gamma$, and that the density function is upper semi-continuous on $\Gamma$, so that
\begin{align}\label{eq-density-V-theta}
	\Theta^n(\vvert V\vvert,x)\geq\frac{1+\cos(\pi-\theta)}{2}
	>0,\quad\forall x\in\Gamma, 
\end{align}
since \eqref{eq-density-V-theta} holds for every $x\in{\rm reg}\Gamma$.

With the nontrivial uniform lower density bound of $V$, we can follow the argument in \cite[Theorem 5.1, Step 2]{LZZ21}.
By using the reflection principle \cref{Lem-ReflectionPrinciple}, we conclude that:
${\rm VarTan}(V,x)$ is non-empty and any $C\in{\rm VarTan}(V,x)$ is a nontrivial, stationary $n$-rectifiable cone in $T_x\overline{\mfR^{n+1}_+}\cong\overline{\mfR^{n+1}_+}$.
Moreover, since we assume $\Gamma$ is a smooth $(n-1)$-manifold in $\p\mfR^{n+1}$, we know that any blow-up of $T$ at $x\in\Gamma$ would be a half $n$-plane in $\p\mfR^{n+1}_+$, and hence $\Theta^n(\vvert {\rm var}(T)\vvert,x)=\frac{1}{2}$, which together with \eqref{eq-density-V-theta} shows that $\Theta^n(\vvert{\rm var}(\pr\Om)\vvert,x)\geq\frac{1}{2}$.
Therefore, ${\rm VarTan}({\rm var}(\pr\Om),x)$ is non-empty and any blow-up limit is a non-trivial rectifiable cone.

For the case $\theta\in(0,\pi/2)$, we consider the varifold (see \cref{Rem-replacement})
$$\tilde V:={\rm var}(\pr\Om)+\cos\theta{\rm var}(\p\mathbf{B}_+\cap\p\mfR^{n+1}_+\setminus T)$$ instead of $V$ in \eqref{defn-Varifold-V}, using the same approach we may conclude that ${\rm VarTan}({\rm var}(\pr\Om),x)$ is non-empty and any blow-up limit is a non-trivial rectifiable cone.

To proceed, we fix any point $x\in\Gamma$.
By the argument above and thanks again to the assumption that $\Gamma\subset\p\mfR^{n+1}$ and $T$ are smooth,
we can find a sequence $\rho_j\searrow0$ as $j\rightarrow\infty$, such that
there exists rectifiable cones ${\rm var}(\tilde M,\psi_1), {\rm var}(\tilde T)$ and ${\rm var}(\tilde\Gamma)$ satisfying (here $\frac{T-x}{\rho_j}\ra\tilde T$ as $j\ra\infty$ is a $n$-dimensional half-space in $\p\mfR^{n+1}_+$ and $\tilde\Gamma$ is the boundary of $\tilde T$ in $\p\mfR^{n+1}_+$)
\begin{align*}
	\frac{1}{\rho_j^n}(\bseta_{x,\rho_j})_\#(\mcH^n\llcorner \pr\Om)
	=\mcH^n\llcorner(\frac{\pr\Om-x}{\rho_j})
	\overset{\ast}{\rightharpoonup}\psi_1\mcH^n\llcorner\tilde M.
\end{align*}
In particular,
by virtue of \eqref{eq-divergence-3} and the stationarity of 
${\rm var}(\tilde M,\psi_1)+\cos(\pi-\theta){\rm var}(\tilde T)$, we have:
\begin{align}
	\int_{\tilde M}\psi_1{\rm div}_{\tilde M}X\rd\mcH^n
	+\cos(\pi-\theta)\int_{\tilde T}{\rm div}_{\p\mfR^{n+1}_+}X\rd \mcH^n
	= -\sin\theta\int_{\tilde\Gamma}X\cdot E_{n+1} \rd\mcH^{n-1},\quad\forall X\in\mathfrak{X}(\mfR_+^{n+1}).
\end{align}
On the other hand, it is clear that we have
\begin{align}
	\int_{\tilde T}{\rm div}_{\p\mfR^{n+1}_+}X\rd\mcH^n =\int_{\tilde\Gamma}X\cdot\bar\nu \rd\mcH^{n-1},\quad\forall X\in\mathfrak{X}_t(\overline{\mfR^{n+1}_+}),
\end{align}
where $\bar\nu$ denotes the outer unit normal of $\tilde T$ along its boundary $\tilde\Gamma$ in $\p\mfR^{n+1}$.

Combining these facts, we find that
the blow-up limit of the triple is $\theta$-stationary in accordance with \cref{Defn-contactangle-triple}.


\subsection{Maximum Principles}\label{Sec-2-10}
We end the preliminary section with the following crucial maximum principles for rectifiable varifolds.

\begin{lemma}[{\cite[Lemma 3]{DM19}}]\label{Lem-MP1}
    Let $M$ be a normalized locally $\mcH^n$-rectifiable set such that ${\rm var}(M,\psi)$ is stationary on $\mfR^{n+1}$.
    If $M$ is a cone (that is, $M=tM$ for every $t>0$), and $M$ is contained in a closed half-space $H$ with $0\in\p H$, then $M=\p H$.
    In particular, $M$ cannot be contained in the convex intersection of two distinct, nonopposite half-spaces containing the origin.
\end{lemma}
Enlightened by the proof of the interior maximum principle \cref{Lem-MP1}, we derive the following boundary maximum principles.
\begin{lemma}\label{Lem-MP2}
    Let $M\subset\overline{\mfR^{n+1}_+}$ be a normalized locally $\mcH^n$-rectifiable set, let $\psi$ be a positive locally $\mcH^n$-integrable function on $M$ with $\psi(x)=\psi(tx)$ for $x\in M$ and $t>0$,
    such that ${\rm var}(M,\psi)$ is a stationary varifold with free boundary in $\overline{\mfR^{n+1}_+}$. If $M$ is a cone (that is, $M=tM$ for every $t>0$), and $M$ is contained in a closed half-space $H$ with $0\in\p H$ and $H$ meets $\p\mfR^{n+1}_+$ orthogonally, then $M=\p H\cap\overline{\mfR^{n+1}_+}$.
    In particular, $M$ cannot be contained in the convex intersection of two distinct, nonopposite half-spaces containing the origin and intersecting $\p\mfR^{n+1}_+$ orthogonally.
\end{lemma}
\begin{proof}[Proof of \cref{Lem-MP2}]
Let $H=\left\{z\in\mfR^{n+1}:z\cdot\nu\leq0\right\}$, where $\nu\in\mathbf{S}^{n-1}\subset\p\mfR^{n+1}_+\cong\mfR^n$.
Given $\varphi\in C_c^\infty([0,\infty))$ with $0\leq\varphi\leq1$, $\varphi(r)=1$ on $[0,\epsilon)$ for some $\ep>0$, and $\varphi'(r)<0$ on $\{0<\varphi<1\}$. Now we set $X(x)=\varphi(\vert x\vert)\nu$ for $x\in\mfR^{n+1}$, clearly $X\in\mathfrak{X}_t(\mfR^{n+1}_+)$, and $\nabla X=\varphi'(\vert x\vert)\nu\otimes\hat x$, where $\hat x=x/\vert x\vert$ if $x\neq0$. Let $\nu_M:M\ra\mathbf{S}^n$ be a Borel vector field such that $T_xM=\nu_M(x)^\perp$ for $\mcH^n$-a.e. $x\in M$. Since $M$ is a cone, we have $\hat x\cdot\nu_M(x)=0$ for $\mcH^n$-a.e. $x\in M$, it follows that
\begin{align*}
    {\rm div}_MX={\rm div}X-\nu_M\cdot\nabla X[\nu_M]=\varphi'(\vert x\vert)(\nu\cdot\hat x-(\nu_M\cdot\nu)(\nu_M\cdot\hat x))=\varphi'(\vert x\vert)(\nu\cdot\hat x),
\end{align*}
and hence by the fact that ${\rm var}(M,\psi)$ is a free boundary stationary varifold in $\overline{\mfR^{n+1}_+}$, we have
\begin{align}
    0=\int_M{\rm div}_MX\psi \rd\mcH^n
    =\int_M\varphi'(\vert x\vert)(\nu\cdot\hat x)\psi(x)\rd\mcH^n(x).
\end{align}
Since $M\subset H$, we know that $\nu\cdot\hat x\leq0$ for every $x\in M$. The arbitrariness of $\ep$ then implies that $\nu\cdot\hat x=0$ for $\mcH^n$-a.e. $x\in M$, and hence $M=\p H\cap\overline{\mfR^{n+1}_+}$.

The rest of the statement in \cref{Lem-MP2} follows easily.
The lemma is thus proved.
\end{proof}
{\color{black}
\begin{lemma}\label{Lem-MP3}
    Given $\theta\in[\pi/2,\pi)$.
    Let $M\subset\overline{\mfR^{n+1}_+}, T\subset\p\mfR^{n+1}_+$ be normalized locally $\mcH^n$-rectifiable sets,
    let $\Gamma\subset\p\mfR^{n+1}_+$ be a normalized locally $\mcH^{n-1}$-rectifiable set.
    Suppose that ${\rm var}(M,\psi_1)$, ${\rm var}(T,\psi_2)$ and ${\rm var}(\Gamma,\psi_3)$ are rectifiable cones (in the sense that $M=tM$ for every $t>0$ and $\psi_1$ is a positive locally $\mcH^n$-integrable function on $M$ with $\psi_1(x)=\psi_1(tx)$ for $x\in M$ and $t>0$),
    such that
    the triple $({\rm var}(M,\psi_1),{\rm var}( T,\psi_2),{\rm var}(\Gamma,\psi_3))$ is $\theta$-stationary as in \cref{Defn-contactangle-triple}.
    If $M$ is contained in a closed half-space $H^-=\{z\in\mfR^{n+1}:z\cdot\nu\leq0\}$ with $0\in\p H^-, \nu\in\mfS^n$,
    then  $\alpha:=\arccos{\left(\nu\cdot(-E_{n+1})\right)}\le \theta$.
    Moreover, if $\alpha=\theta$, then $M=\p H^-\cap\overline{\mfR^{n+1}_+}$.
    
\end{lemma}
\begin{proof}[Proof of \cref{Lem-MP3}]
By definition of $\alpha$, we readily see that there exists a constant vector field $e_1\in T\p\mfR^{n+1}_+$ such that
\begin{align*}
    \nu=\sin\alpha e_1-\cos\alpha E_{n+1}.
\end{align*}

Given $\varphi\in C_c^\infty([0,\infty))$ with $0\leq\varphi\leq1$, $\varphi(r)=1$ on $[0,\epsilon)$ for some $\ep>0$, and $\varphi'(r)<0$ on $\{0<\varphi<1\}$. Now we set $X_1(x)=\varphi(\vert x\vert)e_1$ for $x\in\mfR^{n+1}$, clearly $X_1\in\mathfrak{X}_t(\mfR^{n+1}_+)$, and $\nabla X_1=\varphi'(\vert x\vert)e_1\otimes\hat x$, where $\hat x=x/\vert x\vert$ if $x\neq0$. Let $\nu_M:M\ra\mathbf{S}^n$ be a Borel vector field such that $T_xM=\nu_M(x)^\perp$ for $\mcH^n$-a.e. $x\in M$. Since $M$ is a cone, we have $\hat x\cdot\nu_M(x)=0$ for $\mcH^n$-a.e. $x\in M$, and hence
\begin{align*}
    {\rm div}_MX_1(x)
    =&{\rm div}X_1(x)-\nu_M(x)\cdot\nabla X_1[\nu_M(x)]\\
    =&\varphi'(\vert x\vert)\left(e_1\cdot\hat x-(e_1\cdot\nu_M(x))(\hat x\cdot\nu_M(x))\right)=\varphi'(\vert x\vert)e_1\cdot\hat x.
\end{align*}
Similarly, since $e_1\cdot E_{n+1}=0$,
we have \begin{align*}
    {\rm div}_{\p\mfR^{n+1}_+}X_1
    ={\rm div}X_1-E_{n+1}\cdot\nabla X_1[E_{n+1}]
    =\varphi'(\vert x\vert)e_1\cdot\hat x,\quad\forall x\in\p\mfR^{n+1}_+.
\end{align*}
By virtue of the fact that the triple $({\rm var}(M,\psi_1),{\rm var}( T,\psi_2),{\rm var}(\Gamma,\psi_3))$ is $\theta$-stationary, testing \eqref{formu-1stvariation-3} and \eqref{formu-1stvariation-2} with $X_1$, we find
\begin{align}\label{eq-LemMP3-1}
    &\int_M \varphi'(\vert x\vert)e_1\cdot\hat x\psi_1(x)\rd\mcH^n(x)
    =\cos\theta\int_{T}\varphi'(\vert x\vert)e_1\cdot\hat x\psi_2(x)\rd\mcH^n(x)\notag\\
    =&\cos\theta\int_\Gamma \varphi(\vert x\vert)e_1\cdot\bar\nu(x)\psi_3(x) \rd\mcH^{n-1}(x)
    \leq-\cos\theta\int_\Gamma\varphi(\vert x\vert)\psi_3(x)\rd\mcH^{n-1}(x).
\end{align}
In the last inequality, we used $\theta\in [\frac{\pi}{2},\pi)$.

On the other hand, we consider $X_2(x)=\varphi(\vert x\vert)E_{n+1}$ for $x\in\mfR^{n+1}$, and we have $\nabla X_2=\varphi'(\vert x\vert)E_{n+1}\otimes\hat x$.
Again, since $M$ is a cone, we have
\begin{align*}
    {\rm div}_MX_2(x)
    =&{\rm div}X_2(x)-\nu_M(x)\cdot\nabla X_2[\nu_M](x)\\
    =&\varphi'(\vert x\vert)E_{n+1}\cdot\hat x,
\end{align*}
and also
\begin{align*}
    {\rm div}_{\p\mfR^{n+1}_+}X_2={\rm div}X_2-E_{n+1}\cdot\nabla X_2[E_{n+1}]=0,\quad\forall x\in\p\mfR^{n+1}_+.
\end{align*}
Testing \eqref{formu-1stvariation-3} with $X_2$, we find
\begin{align}\label{eq-LemMP3-2}
    \int_M\varphi'(\vert x\vert)E_{n+1}\cdot\hat x\psi_1(x)\rd\mcH^n(x)
    =-\sin\theta\int_\Gamma \varphi(\vert x\vert)\psi_3(x)\rd\mcH^{n-1}(x).
\end{align}
Recall that $\nu=\sin\alpha e_1-\cos\alpha E_{n+1}$, combining with \eqref{eq-LemMP3-1} and \eqref{eq-LemMP3-2}, we obtain
\begin{align}\label{eq-LemMP3-3}
    &\int_M\varphi'(\vert x\vert)\nu\cdot\hat x\psi_1(x)\rd\mcH^n(x)
    \leq-\sin\alpha\cos\theta\int_\Gamma\varphi(\vert x\vert)\psi_3(x)\rd\mcH^{n-1}(x)\notag\\
    &+\sin\theta\cos\alpha\int_\Gamma\varphi(\vert x\vert)\psi_3(x)\rd\mcH^{n-1}(x)
    =\sin(\theta-\alpha)\int_\Gamma\varphi(\vert x\vert)\psi_3(x)\rd\mcH^{n-1}(x).
\end{align}
Since $M\subset H^-$, we have
\begin{align*}
    \int_M\varphi'(\vert x\vert)\nu\cdot\hat x\psi_1(x)\rd\mcH^n(x)
    \geq0.
\end{align*}
It follows that $\theta\ge \alpha$.

On the other hand, if $\alpha=\theta$, then from the argument above (in particular, \eqref{eq-LemMP3-3}) we know that $\nu\cdot\hat x=0$ for $\mcH^n$-a.e. $x\in M$, which implies $M=\p H^-\cap\overline{\mfR^{n+1}_+}$.
The proof is thus completed.
\end{proof}

We can see from the above proof that the only reason that we have to restrict $\theta\in[\pi/2,\pi)$ is for deriving the inequality \eqref{eq-LemMP3-1}, in other words, if the equality holds in \eqref{eq-LemMP3-1} (consequently, equality holds in \eqref{eq-LemMP3-3}), then we can remove the angle restriction.
In this regard, we have the following maximum principle.
\begin{lemma}\label{Cor-MP3}
Given $\theta\in(0,\pi)$ and a closed half-space $H^-=\{z\in\mfR^{n+1}:z\cdot\nu\leq0\}$ with $0\in\p H^-$, $\nu\in\mfS^n$ (and we let $H^+$ be the antipodal closed half-space).
Let $M\subset\overline{\mfR^{n+1}_+}$ be a normalized locally $\mcH^n$-rectifiable set,
$T=H^+\cap\p\mfR^{n+1}_+,\Gamma=\p H^+\cap\p\mfR^{n+1}_+$.
Suppose that ${\rm var}(M,\psi_1)$ is a rectifiable cone (in the sense that $M=tM$ for every $t>0$ and $\psi_1$ is a positive locally $\mcH^n$-integrable function on $M$ with $\psi_1(x)=\psi_1(tx)$ for $x\in M$ and $t>0$),
such that
the triple $({\rm var}(M,\psi_1),{\rm var}( T),{\rm var}(\Gamma))$ is $\theta$-stationary as in \cref{Defn-contactangle-triple}.
If $M$ is contained in $H^-$,
then $\alpha:=\arccos{\left(\nu\cdot(-E_{n+1})\right)}\le\theta$.
Moreover, if $\alpha=\theta$, then $M=\p H^-\cap\overline{\mfR^{n+1}_+}$.
\end{lemma}
\begin{proof}
We can proceed as the proof of \cref{Lem-MP3} and notice the fact that $e_1=-\bar\nu$ along $\Gamma$
since $T=H^+\cap\p\mfR^{n+1}_+$ and $\Gamma=\p H^-\cap\p\mfR^{n+1}_+$.
In particular,
this implies that \eqref{eq-LemMP3-1} holds as an equality and consequently we obtain the equality in \eqref{eq-LemMP3-3}.
By virtue of the fact that $M\subset H^-$, we can derive the same conclusion as that of \cref{Lem-MP3} for both of the cases when $\pi>\alpha>\theta>0$ and $\alpha=\theta$.
This completes the proof.
\end{proof}

Finally, we are going to exploit the interior strong maximum principle for rectifiable varifolds derived by Sch\"atzle.
\begin{theorem}[{\cite[Theorem 6.2]{Schatzle04}}]\label{Thm-SMP}
    Let $M$ be a normalized locally $\mcH^n$-rectifiable set with distributional mean curvature vector $\mathbf{H}\in L^p(\psi\mcH^n\llcorner M;\mfR^{n+1})$, for some $p>\max\{2,n\}$.
    
    Pick $\nu\in\mathbf{S}^n$, $h_0\in\mfR$, and consider a connected open set $U\subset\nu^\perp$ such that
    \begin{align}
        \varphi(z)=\inf\{h>h_0:z+h\nu\in M\},\quad z\in U,
    \end{align}
    satisfies $\varphi(z)\in(h_0,\infty)$ for every $z\in U$.
    
    If $\eta\in W^{2,p}(U;(h_0,\infty))$ is such that $\eta\leq\varphi$ on $U$ with $\eta(z_0)=\varphi(z_0)$ for some $z_0\in U$, then it cannot be that
    \begin{align}
        -{\rm div}\left(\frac{\nabla\eta}{\sqrt{1+\vert\nabla\eta\vert^2}}\right)(z)\leq\mathbf{H}(z+\varphi(z)\nu)\cdot\frac{-\nabla\varphi(z)+\nu}{\sqrt{1+\vert\nabla\varphi(z)\vert^2}},
    \end{align}
    for $\mcH^n$-a.e. $z\in U$, unless $\eta=\varphi$ on $U$.
\end{theorem}

\
\section{$C^{1,1}$-rectifiability theorem of $\Gamma_{s;\theta}^t$}\label{Sec-3}
Recall $\Gamma_{s;\theta}^t$ defined in \cref{defn-Gamma_s^t}. The main result of this section is the $C^{1,1}$-rectifiability theorem for this set.
\begin{theorem}\label{Thm-C11Rec}
Given $\theta\in(0,\pi)$.
If $\Omega\subset\overline{\mfR^{n+1}_+}$ is a nonempty, bounded, relatively open set with finite perimeter,
then for every $0<t<\infty$, and a.e. $0<s<t$,
there exists a countable collection $\{\mcU_j\}_{j\geq1}$ of compacts subsets of $\Gamma_{s;\theta}^t$ such that $\mcH^n(\Gamma_{s;\theta}^t\setminus\bigcup_{j=1}^\infty\mcU_j)=0$, with each $\mcU_j$ contained in a $C^{1,1}$-hypersurface in $\mfR^{n+1}$.
 Moreover, denoting by $N_\theta$ the gradient of the shifted distance function $\de_\theta$, then $N_\theta\mid_{\mcU_j}$ is Lipschitz for every $j\geq1$.
\end{theorem}

The following fine properties of $\Gamma_{s;\theta}^t$ are crucial for proving the rectifiability result.
\subsection{Fine properties of $\Gamma_{s;\theta}^t$}
Given $\theta\in(0,\pi)$, for any nonempty, bounded, relatively open set with finite perimeter $\Omega\subset\overline{\mfR^{n+1}_+}$, we define $\Gamma_{s;\theta}^t$ and $\Gamma_{s;\theta}^+$ as in \cref{defn-Gamma_s^t}.
Recall that for any $y\in\mfR^{n+1}$, the shifted distance function from $\pr\Om$ is defined as $\de_\theta(y)=\sup_{r\geq0}\{r:B_r(y-r\cos\theta E_{n+1})\cap \pr\Om=\varnothing\}$.
Following \cite[Definition 4.1]{Fed59}, we define the shifted unique point projection mapping
as follows.
\begin{definition}
    \normalfont
Given $\theta\in(0,\pi)$, for any nonempty, bounded, relatively open set $\Omega\subset\overline{\mfR^{n+1}_+}$, let ${\rm Unp}_\theta(\pr\Om)$ be the set of points $y\in\mfR^{n+1}$ for which there exists a unique point of $\pr\Om$ nearest to $y$ with respect to the shifted distance function $\de_\theta$, and the map
\begin{align}\label{defn-unp}
    \xi_\theta:{\rm Unp}_\theta(\pr\Om)\ra \pr\Om
\end{align}
associates with $y\in{\rm Unp}_\theta(\pr\Om)$ the unique $x\in \pr\Om$ such that $\de_\theta(y)={\rm dist}(x,y-\de_\theta(y)\cos\theta E_{n+1})$.
\end{definition}
Our first observation is that $\Gamma_{s;\theta}^t\subset{\rm Unp}_\theta(\pr\Om)$, we refer to \cref{figure3} for illustration.
	\begin{figure}[h]
	\centering
	\includegraphics[height=8cm,width=13cm]{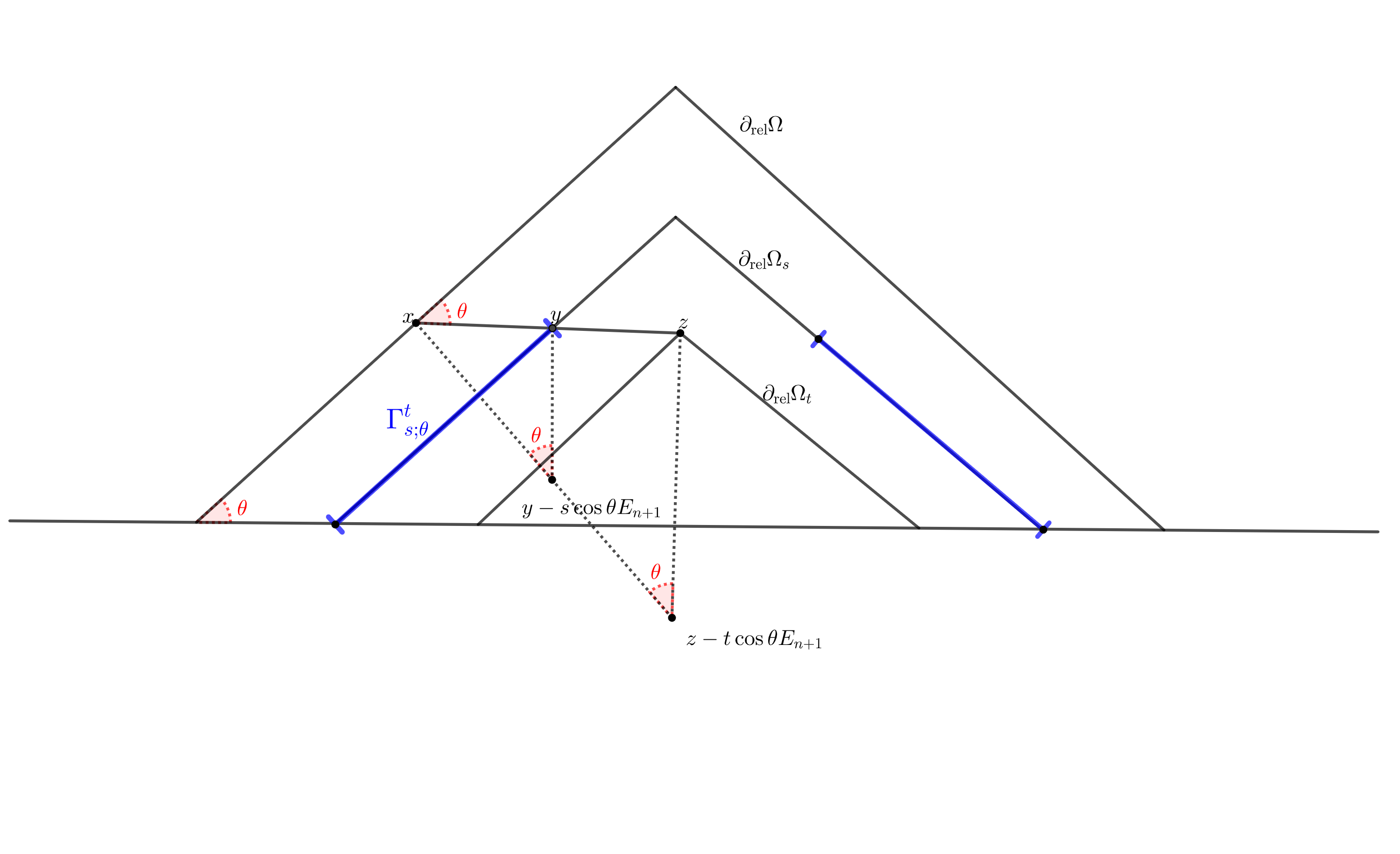}
	\caption{$\Gamma_{s;\theta}^t$}
	\label{figure3}
\end{figure}

\begin{lemma}\label{Lem-Unp}
    Given $\theta\in(0,\pi)$, for any nonempty, bounded, relatively open set with finite perimeter $\Omega\subset\overline{\mfR^{n+1}_+}$ and for every $0<s<t<\infty$, let
    $\Gamma_{s;\theta}^t$ be as in \cref{defn-Gamma_s^t}.    
    Then, for any $y\in\Gamma_{s;\theta}^t$, it has a unique point projection onto $\pr\Om$ with respect to $\de_\theta$, which reads as $\xi_\theta(y)=x$; in other words, $\Gamma_{s;\theta}^t\subset{\rm Unp}_\theta(\pr\Om)$.
\end{lemma}
\begin{proof}
By definition, for any $y\in \Gamma_{s;\theta}^{t}$,
there exists $x=\gamma_y(0)\in\pr\Om$, $z=\gamma_y(t)\in\pr\Om_{t}$.
Consider the geodesic $\gamma_y(r)-r\cos\theta E_{n+1}$ defined on $r\in[0,t]$,
by \eqref{eq-dist-theta} there holds
\begin{align*}
    \de(\gamma_y(r)-r\cos\theta E_{n+1})
    =\de_\theta(\gamma_y(r))
    =r\quad\forall r\in[0,t],
\end{align*}
which means $\gamma_y(r)-r\cos\theta E_{n+1}$ is a unit-speed line segment, and it is easy to see the following claim holds.

{\bf Claim. }$\vert x-(y-s\cos\theta E_{n+1})\vert=s, \vert y-s\cos\theta E_{n+1}-(z-t\cos\theta E_{n+1})\vert=t-s$.

Now we show that $x$ is uniquely determined by $y$. Since $y\in\pr\Omega_s$, if there exists $x'\neq x\in\pr\Om$ such that $\vert x'-(y-s\cos\theta E_{n+1})\vert=s$,
then by the triangle inequality
we have
	\begin{align*}
		\vert x'-(z-t\cos\theta E_{n+1})\vert< \vert x'-(y-s\cos\theta E_{n+1})\vert+\vert y-z+(t-s)\cos\theta E_{n+1}\vert=s+t-s=t,
	\end{align*}
	contradicts to the fact that $z\in\pr\Omega_t$.
 Therefore, we have showed that $x$ is the unique point of $\pr\Om$ nearest to $y$ with respect to $\de_\theta$; that is, $\xi_\theta(y)=x$.
\end{proof}
Once $0<s<t<\infty$ and $y\in\Gamma_{s;\theta}^t$ are fixed, it follows from \cref{Lem-Unp} that the geodesic (constant-speed line segment) $\gamma_y:[0,t]\ra\mfR^{n+1}$ as in \cref{defn-Gamma_s^t} is unique.
Moreover, for the geodesic $\gamma_y(r)-r\cos\theta E_{n+1}$ defined on $r\in[0,t]$, we see from the proof of \cref{Lem-Unp} that any point on it has a unique point projection onto $\pr\Om$ with respect to the distance function $\de$, which reads as $$\xi(\gamma_y(r)-r\cos\theta E_{n+1})=x,\quad\forall r\in[0,t).$$

Equipped with the fact that $\Gamma_{s;\theta}^t\subset{\rm Unp}_\theta(\pr\Om)$, we can explore the shifted distance function $\de_\theta$ on $\mfR^{n+1}$ and the unique point projection mapping $\xi_\theta$ on $\Gamma_{s;\theta}^t$. We collect the fine properties of $\delta_\theta$ and $\xi_\theta$ as follows.
\begin{lemma}\label{Lem-u-xi}
	Given $\theta\in(0,\pi)$, for any nonempty, bounded, relatively open set with finite perimeter $\Om\subset\overline{\mfR^{n+1}_+}$, the following statements hold:
	\begin{enumerate}
		\item \label{Lem-u-xi-item1}
		$\de_\theta$ is a Lipschitz function on $\mfR^{n+1}$ with Lipschitz constant at most $\frac{1}{1-\vert\cos\theta\vert}$, i.e., for any $x,y\in \mfR^{n+1}$,
		\begin{align*}
			\vert \de_\theta(y)-\de_\theta(x)\vert\leq\frac{1}{1-\vert\cos\theta\vert}\vert x-y\vert.
		\end{align*}
		\item
		\label{Lem-u-xi-item2}
		For $0\leq r_1<r_2$, and for any $y\in\mfR^{n+1}$, the strictly inclusion holds:
		\begin{align*}
		    B_{r_1}(y-r_1\cos\theta E_{n+1})\subset B_{r_2}(y-r_2\cos\theta E_{n+1}).
		\end{align*}
		\item
		\label{Lem-u-xi-item3}
		$N_\theta(y):=\nabla \de_\theta(y)$ exists for $\mcL^{n+1}$-a.e. $y\in\mfR^{n+1}$. Moreover, when it exists, there holds
		\begin{align}\label{ineq-mod-Ntheta}
		    \frac{1}{1+\vert\cos\theta\vert}\leq\vert N_\theta(y)\vert\leq\frac{1}{1-\vert\cos\theta\vert}.
		\end{align}
		In particular, as $\theta=\pi/2$, $\vert N_\theta(y)\vert=1$ when it exists. 
		\item \label{Lem-u-xi-item4}
		For $0<s<t$,
		$\xi_\theta$ is continuous on $\Gamma_{s;\theta}^t$.
	\end{enumerate}
\end{lemma}

\begin{proof}
	{\bf(1)} We fix any $x,y\in\mfR^{n+1}$. Since $\pr\Om$ is compact in $\mfR^{n+1}$, we can take $a\in\pr\Om$ such that $\de_\theta(x)
	=\vert a-(x-\de_\theta(x)\cos\theta E_{n+1})\vert$. We assume without loss of generality that $\de_\theta(y)\geq \de_\theta(x)$, then by the triangle inequality, we find
	\begin{align*}
		\vert\de_\theta(y)-\de_\theta(x)\vert
		=&\de_\theta(y)-\de_\theta(x)\notag\\
		\leq&\vert a-(y-\de_\theta(y)\cos\theta E_{n+1})\vert-\vert a-(x-\de_\theta(x)\cos\theta E_{n+1})\vert\notag\\
		\leq&\vert x-y+(\de_\theta(y)-\de_\theta(x))\cos\theta E_{n+1}\vert\notag\\
		\leq&\vert x-y\vert+\vert \de_\theta(y)-\de_\theta(x)\vert\vert \cos\theta\vert,
	\end{align*}
	we rearrange this to see
	\begin{align}
	    \vert \de_\theta(y)-\de_\theta(x)\vert\leq\frac{1}{1-\vert\cos\theta\vert}\vert x-y\vert.
	\end{align}
	This completes the proof of {\bf(1)}.
	
	\noindent{\bf(2)} This amounts to be a simple observation due to the triangle inequality.
	Indeed, for any $z\in\p B_{r_1}(y-r_1\cos\theta E_{n+1})$, if $\vert z-(y-r_2\cos\theta E_{n+1})\vert\geq r_2$, then we have
 \begin{align*}
	    {\rm dist}(z,y-r_1\cos\theta E_{n+1})+{\rm dist}(y-r_1\cos\theta E_{n+1},y-r_2\cos\theta E_{n+1})\geq{\rm dist}(z,y-r_2\cos\theta E_{n+1}),
	\end{align*}
	namely,
	\begin{align*}
	    r_1+(r_2-r_1)\vert\cos\theta\vert\geq r_2,
	\end{align*}
	which leads to a contradiction since $\theta\in(0,\pi)$ and completes the proof of {\bf (2)}.
	
	\noindent{\bf (3)}
	By virtue of the Rademacher's theorem, $\de_\theta$ is differentiable at $\mcL^{n+1}$-a.e. $y\in\mfR^{n+1}$. When $N_\theta(y)$ exists, we may assume that there exists a unique $x\in\pr\Om$, such that $\de_\theta(y)=\vert x-(y-\de_\theta(y)\cos\theta E_{n+1})\vert$.
	
	{\bf Claim 1. }$\nabla \de_\theta(y)\parallel x-(y-\de_\theta(y)\cos\theta E_{n+1}).$
	
	Indeed, thanks to \eqref{eq-dist-theta},
	we have the following relation of distance and shifted distance:
	\begin{align}\label{eq-delta&delta-theta}
	    \delta(y-\delta_\theta(y)\cos\theta E_{n+1})=\delta_\theta(y).
	\end{align}
	Notice that by virtue of \cite[Theorem 4.8 (3)]{Fed59}, $\left(\nabla\delta\right)(y-\delta_\theta(y)\cos\theta E_{n+1})=\frac{x-(y-\delta_\theta(y)\cos\theta E_{n+1})}{\delta_\theta(y)}$.
	Differentiating both sides of \eqref{eq-delta&delta-theta}, we get the claim.
	
	{\bf Claim 2. }For any $0< t\leq1$, $\de_\theta\left(y+t(x-y)\right)=(1-t)\de_\theta(y)$.
	
	Observe that $y+t(x-y)-(1-t)\de_\theta(y)\cos\theta E_{n+1}=y-\de_\theta(y)\cos\theta E_{n+1}+t\left(x-(y-\de_\theta(y)\cos\theta E_{n+1})\right)$, and trivially we have $\vert x-{\rm RHS}\vert=(1-t)\de_\theta(y)$, since ${\rm dist}(x,y-\de_\theta(y)\cos\theta E_{n+1})=\de_\theta(y)$. Therefore, we have: $\de_\theta(y+t(x-y))\leq(1-t)\de_\theta(y)$.
	
	On the other hand, since $1/(1-t)>1$, we know that $$\bar B_{(1-t)\de_\theta(y)}\left(y+t(x-y)-(1-t)\de_\theta(y)\cos\theta E_{n+1}\right)\subset \bar B_{\de_\theta(y)}\left(y-\de_\theta(y)\cos\theta E_{n+1}\right)$$ with the only common point $x$, and it follows from {\bf(2)} that for any $r<(1-t)\de_\theta(y)$, $B_r(y+t(x-y)-r\cos\theta E_{n+1})\cap\pr\Om=\varnothing$, which implies $\de_\theta(y+t(x-y))=(1-t)\de_\theta(y)$ and proves the claim.

    Now we are ready to carry out the proof of \eqref{ineq-mod-Ntheta}, we abbreviate $N_\theta(y)/\vert N_\theta(y)\vert$ by $\hat N_\theta(y)$.
    
    By {\bf Claim 1} we know that $x-(y-\de_\theta(y)\cos\theta E_{n+1})=-\de_\theta(y)\hat N_\theta(y)$, and hence $x-y=-\de_\theta(y)(\hat N_\theta(y)+\cos\theta E_{n+1})$.
    Notice that the Taylor expansion of $\de_\theta(y+t(x-y))$ at $y$ gives
    \begin{align*}
        \de_\theta(y+t(x-y))=\de_\theta(y)+tN_\theta(y)\cdot \left(x-y\right)+o(t),
    \end{align*}
    which, by {\bf Claim 2} and the above observation, reads as
    \begin{align*}
        -t\de_\theta(y)=tN_\theta(y)\cdot\left(-\de_\theta(y)\left(\hat N_\theta(y)+\cos\theta E_{n+1}\right)\right)+o(t),
    \end{align*}
    sending $t\searrow0^+$ and rearranging, we thus find
    \begin{align}\label{eq-mod-Ntheta}
        \vert N_\theta(y)\vert=\frac{1}{1+\cos\theta\hat N_\theta(y)\cdot E_{n+1}},
    \end{align}
   \eqref{ineq-mod-Ntheta} follows easily.
    
	For {\bf(4)}, suppose on the contrary that there exists some $\ep>0$ and a sequence of points $\{y_i\}_{i=1}^\infty\subset\Gamma_{s;\theta}^t$,
	converges to $y\in\Gamma_{s;\theta}^t$, such that $\vert \xi_\theta(y)-\xi_\theta(y_i)\vert\geq\ep$ for every large $i$.
	
    By definition, for each $i$, we have, for $i$ large, there holds
	\begin{align}
	   \vert\xi_\theta(y_i)-(y_i-s\cos\theta E_{n+1})\vert =\de_\theta(y_i)
	    =s.
	\end{align}
	Using the triangle inequality and the fact that $y_i$ converges to $y$, we find
	\begin{align*}
		\vert \xi_\theta(y_i)-(y-s\cos\theta E_{n+1})\vert
		\leq&
		\vert\xi_\theta(y_i)-(y_i-s\cos\theta E_{n+1})\vert+\vert y_i-y\vert\notag\\
		=&s+\vert y_i-y\vert<s+\ep.
	\end{align*}
	This means, all the points $\{\xi_\theta(y_i)\}_i$ are lying in $\pr\Om\cap B_{s+\ep}(y-s\cos\theta E_{n+1})$, which is a bounded subset of the compact set $\pr\Om$, and hence by passing to a subsequence, we can assume that $\{\xi_\theta(y_i)\}_i$ converges to some point $x\in\pr\Om$. But then, since $\de_\theta$ is continuous, we have
	\begin{align*}
		s=
		\de_\theta(y)
  =\lim_{i\ra\infty}\de_\theta(y_i)
  =\lim_{i\ra\infty}\vert\xi_\theta(y_i)-(y_i-s\cos\theta E_{n+1})\vert
		=\vert x-(y-s\cos\theta E_{n+1})\vert,
	\end{align*} 
	which implies that $x=\xi_\theta(y)$ since we have proved that $y\in\Gamma_{s;\theta}^t\subset{\rm Unp}_\theta(\pr\Om)$ in \cref{Lem-Unp}. However, this contradicts to the assumption that
	\begin{align*}
		\vert x-\xi_\theta(y)\vert
  =\lim_{i\ra\infty}\vert\xi_\theta(y_i)-\xi_\theta(y)\vert\geq\ep,
	\end{align*}
	and hence completes the proof.
\end{proof}

\begin{remark}
	\normalfont
	When $\Om$ is contained in a Euclidean space, similar results are included in \cite[4.8(1)(2)(3)(4)]{Fed59}.
\end{remark}

With the help of \cref{Lem-Unp} and \cref{Lem-u-xi}, we can fully explore the fine properties of $\Gamma_{s;\theta}^t$ and $\Gamma_{s;\theta}^+$, which are well understood in the Euclidean case, see \cite[Theorem1]{DM19}.

\begin{proposition}\label{Prop-Gammast}
	Given $\theta\in(0,\pi)$.
 For any nonempty, bounded, relatively open set with finite perimeter $\Om\subset\overline{\mfR^{n+1}_+}$ and for every $0<s<t<\infty$,
 there holds
	\begin{enumerate}
		\item \label{prop3.1-item2}For $s<t_1<t_2$, $\Gamma_{s;\theta}^{t_2}\subset\Gamma_{s;\theta}^{t_1}$. In particular, $\Gamma_{s;\theta}^+=\lim_{t\ra s^+}\Gamma_{s;\theta}^t$.
		{\color{black}\item \label{prop3.1-item3}$\Gamma_{s;\theta}^t$ is a compact set in $\overline{\mfR^{n+1}_+}$.}
		\item \label{prop3.1-item4}At any $y\in\Gamma_{s;\theta}^t$, $\Gamma_{s;\theta}^t$ is bounded by two mutually tangent balls in $\mfR^{n+1}$ with radii $s$ and $t-s$.
		\item \label{prop3.1-item5}$\de_\theta$ is differentiable at every $y\in\Gamma_{s;\theta}^t$.
	\end{enumerate}
\end{proposition}

\begin{proof}
	\noindent\textbf{(1)}
	The first part of the statement follows easily from the definition of $\Gamma_{s;\theta}^t$, while by virtue of the inclusion, it is apparent that $\Gamma_{s;\theta}^+=\lim_{t\ra s^+}\Gamma_{s;\theta}^t$.
	\vspace{0.3cm}
	
	\noindent{\textbf{(2)}} It suffice to prove that $\Gamma_{s;\theta}^t$ is a closed set, i.e., if a sequence of points $\{y_i\}_{i=1}^\infty\subset\Gamma_{s;\theta}^t$ converges to $y$, then it must be that $y\in\Gamma_{s;\theta}^t$.
	
	By definition of $\Gamma_{s;\theta}^t$, for each $y_i$, there exists corresponding points $x_i\in\pr\Om, z_i\in\pr\Om_t$. By \cref{Lem-u-xi}, $\xi_\theta$ is continuous on $\Gamma_s^t$, and hence we have: $\{x_i\}_{i=1}^\infty$ is a Cauchy sequence 
 in $\pr\Om$. Notice that $\pr\Om$ is closed in $\mfR^{n+1}$, and hence $\{x_i\}_{i=1}^\infty$ converges to some $x\in\pr\Om$. Similarly, $\{z_i\}_{i=1}^\infty$ converges to some $z\in\pr\Om_t$. 
	
	By continuity, we have
	\begin{align*}
		\de_\theta(y)=\lim_{i\ra\infty}\de_\theta(y_i)=s,\quad
  \vert x-(y-s\cos\theta E_{n+1})\vert=\lim_{i\ra\infty}\vert \xi_\theta(y_i)-(y_i-s\cos\theta E_{n+1})\vert=s.
	\end{align*}
 Similarly, we deduce that $\vert y-s\cos\theta E_{n+1}-(z-t\cos\theta E_{n+1})\vert=t-s$.
By using the triangle inequality again, we find
	\begin{align*}
		t=\de_\theta(z)
		\leq&\vert x-(z-t\cos\theta E_{n+1})\vert\notag\\
		\leq&\vert x-(y-s\cos\theta E_{n+1})\vert+\vert y-z+(t-s)\cos\theta E_{n+1}\vert=t,
	\end{align*}
	it is easy to see that the line segment joining $x$ and $z-t\cos\theta E_{n+1}$ must contain $y-s\cos\theta E_{n+1}$, and hence one may verify that $y\in\Gamma_{s;\theta}^t$ by definition, which completes the proof of ({\bf 2}).
	\vspace{0.3cm}
	
	\noindent\textbf{(3)} can be deduced from the definition of $\Gamma_{s;\theta}^t$ and \cref{Lem-Unp}.
 Indeed, it is easy to see that for every $y\in\Gamma_{s;\theta}^t$,
		\begin{align}\label{eq-mutuallytangentballs}
			\{y-s\cos\theta E_{n+1}\}=\p B_{t-s}(z-t\cos\theta E_{n+1})\cap\p B_s(x),
		\end{align}
		and hence
  $\Gamma^t_{s;\theta}-s\cos\theta E_{n+1}$ is trapped between
	$B_{t-s}(z-t\cos\theta E_{n+1})$ and $B_s(x)$ at $y-s\cos\theta E_{n+1}$.
	This in turn implies that $\Gamma^t_{s;\theta}$ is trapped between two mutually tangent balls with radii $s$ and $t-s$ at every of its points.
	
	\noindent\textbf{(4)} is a direct consequence of the fact that $y\in\Gamma_{s;\theta}^t\subset{\rm Unp}_\theta(\pr\Om)$, which we proved in \cref{Lem-Unp}.
\end{proof}


\subsection{$C^{1,1}$-rectifiability theorem and its consequences}
\begin{proof}[Proof of \cref{Thm-C11Rec}]
We follow closely the proof in \cite[Proof of Theorem 1: Step 1]{DM19} with modifications to our shifted distance function $\de_\theta$.
	Recall that we denote by $N_\theta(y)$ the gradient of $\de_\theta$ at $y$, which exists at every $y\in\Gamma_{s;\theta}^t$ thanks to \cref{Prop-Gammast}(4), and we denote the normalized vector $N_\theta(y)/\vert N_\theta(y)\vert$ simply by $\hat N_\theta(y)$.
	
	\noindent{\bf Step 1. }$C^1$-rectifiability of $\Gamma_{s;\theta}^t$.

	First we estimate $\left\vert N_\theta(y)\cdot(y'-y)\right\vert$ for any $y,y'\in\Gamma_{s;\theta}^t$.
	By virtue of \cref{Lem-u-xi}\eqref{Lem-u-xi-item3} and \cref{Prop-Gammast}, we get an explicit expression of $N_\theta(y)$ by $x,y,z$.
Indeed, by virtue of {\bf Claim 1} in the proof of \cref{Lem-u-xi}\eqref{Lem-u-xi-item3}, we find
	\begin{align}\label{defn-Ntheta-y}
	    \hat N_\theta(y)
	    =\frac{-x+(z-t\cos\theta E_{n+1})}{\vert -x+(z-t\cos\theta E_{n+1})\vert}.
	\end{align}
	This, together with {\bf Claim 2} in the proof of \cref{Lem-u-xi}\eqref{Lem-u-xi-item3},
	implies the following fact:
	for any $y\in\Gamma_{s;\theta}^t$, and for $r\in[-s,t-s]$, there holds
	\begin{align}
	    y-s\cos\theta E_{n+1}+r\hat N_\theta(y)\in\pr\Om_{s+r}-(s+r)\cos\theta E_{n+1},
	\end{align}
	here we adopt the convention that $\pr\Om_0=\pr\Om$. We rearrange this to see
	\begin{align}\label{eq-DM-3-11}
	    y+r(\hat N_\theta(y)+\cos\theta E_{n+1})\in\pr\Om_{s+r}.
	\end{align}
	Using \eqref{eq-DM-3-11} and recalling \eqref{eq-dist-theta}, for $y,y'\in\Gamma_{s;\theta}^t$, we have
	\begin{align}
	    s^2
	    \leq&\vert y-s(\hat N_\theta(y)+\cos\theta E_{n+1})-(y'-s\cos\theta E_{n+1})\vert^2\notag\\
	    =&s^2-2s\hat N_\theta(y)\cdot(y-y')+\vert y-y'\vert^2,
	\end{align}
	and also
	\begin{align}
	    (t-s)^2
	    \leq&\vert \left(y+(t-s)(\hat N_\theta(y)+\cos\theta E_{n+1})-t\cos\theta E_{n+1}\right)-(y'-s\cos\theta E_{n+1})\vert^2\notag\\
	    =&\vert y+(t-s)\hat N_\theta(y)-y'\vert^2
	    =(t-s)^2+2(t-s)\hat N_\theta(y)\cdot(y-y')+\vert y-y'\vert^2.
	\end{align}
	Combining these facts and using \eqref{ineq-mod-Ntheta}, we obtain the estimate
	\begin{align}\label{eq-DM-3-12}
	    \vert N_\theta(y)\cdot(y'-y)\vert
	    \leq&\vert N_\theta(y)\vert\max\{\frac{1}{s},\frac{1}{t-s}\}\frac{\vert y-y'\vert^2}{2}\notag\\
	    \leq&\max\{\frac{1}{s},\frac{1}{t-s}\}\frac{\vert y-y'\vert^2}{2(1-\vert\cos\theta\vert)}\quad\text{for all }y,y'\in\Gamma_{s;\theta}^t.
	\end{align}
	By \cref{Prop-Gammast}, $N_\theta$ is continuous on $\Gamma_{s;\theta}^t$, and hence we have the pair $(\de_\theta,N_\theta)\in C^0(\Gamma_{s;\theta}^t;\mfR\times\mfR^{n+1})$, satisfying the key estimate \eqref{eq-DM-3-12}.
	Observe that
	\begin{align}\label{eq-C1-Whitney}
		&\limsup_{\ep\ra0^+}\{\frac{\vert \de_\theta(y')-\de_\theta(y)-N_\theta(y)\cdot(y'-y)\vert}{\vert y'-y\vert}:0<\vert y'-y\vert\leq\ep,\quad y',y\in\Gamma_{s;\theta}^t\}\notag\\	\leq&\limsup_{\ep\ra0^+}\left\{\frac{\max\{\frac{1}{2(t-s)},\frac{1}{2s}\}\vert y-y'\vert^2}{\left(1-\vert\cos\theta\vert\right)\vert y'-y\vert}:0<\vert y'-y\vert\leq\ep,\quad y',y\in\Gamma_{s;\theta}^t\right\}
		=0,
	\end{align}
	where in the inequality we used the fact that $\de_\theta(y')=\de_\theta(y)=s$ and \eqref{eq-DM-3-12}.
	In particular, with this in force, we can use the $C^1$-Whitney's extension theorem to find that there exists $\phi_\theta\in C^1(\mfR^{n+1})$ such that $(\phi_\theta,\nabla\phi_\theta)=(\de_\theta,N_\theta)$ on $\Gamma_{s;\theta}^t$.
	Moreover, by \eqref{defn-Ntheta-y} we know that $N_\theta(y)\neq0$, and hence we can use the $C^1$-Implicit function theorem for $\phi_\theta$ to find: for every $y\in\Gamma_{s;\theta}^t$, there exists an open set $U\subset\mfR^n$, and a $C^1$-function $\psi:U\ra\mfR^1$, such that $\Gamma_{s;\theta}^t\subset(x',\psi(x'))$ on $U$, i.e., $\Gamma_{s;\theta}^t$ lies in the $C^1$-image of $\Psi:U\subset\mfR^n\ra\mfR^{n+1}$, given by $\Psi(x')=(x',\psi(x'))$. On the other hand, using the coarea formula and invoking \eqref{ineq-mod-Ntheta}, we obtain
	\begin{align}
	    \frac{\vert\Om\vert}{1-\vert\cos\theta\vert}
	    \geq\int_\Om\vert\nabla \de_\theta\vert \rd\mcL^{n+1}=\int_0^\infty\mcH^n(\pr\Om_s)\rd s,
	\end{align}
	since $\Om$ is bounded, this implies, for a.e. $s>0$, $
	\mcH^n(\pr\Om_s)<\infty$, it follows that $\mcH^n(\Gamma_{s;\theta}^t)<\infty$ for a.e. $s>0$.
    In particular, these facts yield the $\mcH^n$-rectifiability of $\Gamma_{s;\theta}^t$.
	
	\vspace{0.3cm}
	\noindent{\bf Step 2.} $N_\theta$ is tangentially differentiable along $\Gamma_{s;\theta}^t$ at $\mcH^n$-a.e. $y\in\Gamma_{s;\theta}^t$.
	
	As presented in \cite{DM19}, the key point of this step is to show that $N_\theta$ is Locally Lipschitz on $\Gamma_{s;\theta}^t$. Indeed, we will show that for any $\Gamma_{s;\theta}^t$ having finite $\mcH^n$-measure, it can be covered, up to a $\mcH^n$-negligible set, by compact sets $\{\mcU_j\}_{j\in\mathbb{N}}$,  such that $N_\theta\mid_{\mcU_j}$ is Lipschitz.
	
	We first construct $\{\mcU_j\}$, let $\mathcal{C}(N,\rho)=\{z+hN:z\in N^\perp,\vert z\vert<\rho, \vert h\vert<\rho\}$ denote the open cylinder in $\mfR^{n+1}$, centered at the origin with axis along $N\in\mfS^n$, radius $\rho>0$ and height $2\rho$.
	By \cref{Prop-Gammast} and the $C^1$-rectifiability of $\Gamma_{s;\theta}^t$, we know that $\Gamma_{s;\theta}^t$ admits an approximate tangent plane at $\mcH^n$-a.e. of its points and this plane is then exactly $\left\{\hat N_\theta(y)\right\}^\perp$, which is a $n$-dimensional affine plane in $\mfR^{n+1}$, i.e., 
	\begin{align*}
		T_y\Gamma_{s;\theta}^t=\left\{\hat N_\theta(y)\right\}^\perp \quad\text{for }\mcH^n\text{-a.e. }y\in\Gamma_{s;\theta}^t. 
	\end{align*}
	By \cite[Theorem 10.2]{Mag12} and notice that for any fixed $\rho$, there exists $0<\rho_1<\rho_2$ such that $B_{\rho_1}\subset\mathcal{C}(\hat N_\theta(y),\rho)\subset B_{\rho_2}$, we have
	\begin{align*}
		\lim_{\rho\ra0^+}\frac{\mcH^n\left(\Gamma_{s;\theta}^t\cap\left(y+\mathcal{C}\left(\hat N_\theta\left(y\right),\rho\right)\right)\right)}{\om_n\rho^n}=1,\quad \text{for }\mcH^n\text{-a.e. }y\in\Gamma_{s;\theta}^t.
	\end{align*}
	For a sequence $\{\rho_j\}_{j\in\mathbb{N}}$ such that $\rho_j\ra0$ as $j\ra\infty$, we set
	\begin{align*}
		f_j(y):=\frac{\mcH^n\left(\Gamma_{s;\theta}^t\cap\left(y+\mathcal{C}\left(\hat N_\theta\left(y\right),\rho_j\right)\right)\right)}{\om_n\rho_j^n},
	\end{align*}
	then $f_j\ra1$ for $\mcH^n$-a.e. $y\in\Gamma_{s;\theta}^t$. By Egoroff's theorem and \cite[Lemma 1.1]{EG15}, there exists a compact set $\mcU_1\subset\Gamma_{s;\theta}^t$ such that $f_j\ra1$ uniformly on $\mcU_1$ and $\mcH^n(\Gamma_{s;\theta}^t\setminus \mcU_1)<\frac{1}{2}\mcH^n(\Gamma_{s;\theta}^t)$.
	For $\Gamma_{s;\theta}^t\setminus \mcU_1$, we can use Egoroff's theorem again to find a compact set $\mcU_2\subset\Gamma_{s;\theta}^t\setminus \mcU_1$ such that $f_j\ra1$ uniformly on $\mcU_2$ and $\mcH^n\left( \Gamma_{s;\theta}^t\setminus\left(\mcU_1\cup \mcU_2\right)\right)<\frac{1}{2^2}\mcH^n(\Gamma_{s;\theta}^t)$. Using an inductive argument, we obtain a sequence of compact sets $\left\{\mcU_j\right\}_{j=1}^\infty$ such that $\mcH^n(\Gamma_{s;\theta}^t\setminus\bigcup_{j=1}^\infty\mcU_j)=0$ with $f_j\ra1$ uniformly on each $\mcU_j$. Consequently, we have 
	\begin{align}\label{defn-muj-star}
		\mu_j^\ast(\rho):=\sup_{y\in \mcU_j}\left\vert 1-\frac{\mcH^n\left(\Gamma_{s;\theta}^t\cap(y+\mathcal{C}(\hat N_\theta(y),\rho)\right)}{\om_n\rho^n}\right\vert\ra0 \quad\text{as }\rho\ra0^+.
	\end{align}
This shows that $\Gamma_{s;\theta}^t$ can be covered by a countable union of compact sets, up to a $\mcH^n$-negligible set.

Fix any $y\in \mcU_j\subset\Gamma_{s;\theta}^t$,
we know from the $C^1$-Implicit function theorem that $\mcU_j\subset\Gamma_{s;\theta}^t$ is the graph of a $C^1$-function $\psi_j(\cdot):\mfR^n\ra\mfR^1$ in a neighborhood of $y$. Therefore, up to further subdivision of $\mcU_j$, we may assume that each $\mcU_j$ satisfies: for each $y\in\mcU_j$, there exists
\begin{align}\label{defn-psij-1}
    \psi_j\in C^1(\{\hat N_\theta(y)\}^\perp),\quad\psi_j(0)=0,\quad\nabla\psi_j(0)=0,\quad\vvert\nabla\psi_j\vvert_{C^0(\{\hat N_\theta(y)\}^\perp)}\leq1,
\end{align}
such that, if $\mcU_j'$ denotes the projection of $\mcU_j$ on $\{\hat N_\theta(y)\}^\perp\cap\{\vert z\vert<\rho_j\}$, then
\begin{align}\label{defn-psij-2}
    \mcU_j\cap(y+\mathcal{C}(\hat N_\theta(y),\rho_j))=\Gamma_{s;\theta}^t\cap(y+\mathcal{C}(\hat N_\theta(y),\rho_j))=y+\{z+\psi_j(z)\hat N_\theta(y):z\in\mcU_j'\},
\end{align}
here $\rho_j$, $\psi_j$ depend on the choice of the point $y\in\mcU_j$. 
Moreover, if we set
\begin{align}\label{defn-muj}
	    \mu_j(\rho):=\max\left\{\mu^\ast_j(\rho),\max_{\vert z\vert\leq\rho}\vert\nabla\psi_j(z)\vert\right\},\quad\rho\in(0,\rho_j],
	\end{align}
	then $\mu_j(\rho)\ra0$ as $\rho\ra0^+$ by \eqref{defn-muj-star} and the continuity of $\nabla\psi_j$. This completes the construction of the covering $\{\mcU_j\}_{j\in\mathbb{N}}$.

To proceed, we need to show that $N_\theta$ is Lipschitz on each $\mcU_j$, i.e., for any $y_1,y_2\in\mcU_j$, there exists $C_j>0$ for each $j$, such that
\begin{align}\label{ineq-N-Lip}
    \vert N_\theta(y_1)-N_\theta(y_2)\vert\leq C_j\vert y_1-y_2\vert.
\end{align}
Since $\mcU_j$ can be written as a $C^1$-graph, with \cref{Prop-Gammast}(3), \eqref{eq-DM-3-12} , \cref{Lem-u-xi}(3) and \eqref{defn-muj} in force, we can adopt exactly the same approach in \cite[(3-16)]{DM19}
to find
\begin{align}\label{ineq-Lip-hatN}
    \vert\hat N_\theta(y_1)-\hat N_\theta(y_2)\vert\leq \tilde C_j\vert y_1-y_2\vert.
\end{align}
To see that
\eqref{ineq-N-Lip} holds, we use the triangle inequality to obtain
\begin{align*}
    \vert N_\theta(y_1)-N_\theta(y_2)\vert
    \leq\vert N_\theta(y_1)\Big\vert\vert\hat N_\theta(y_1)-\hat N_\theta(y_2)\Big\vert+\Big\vert\left\vert N_\theta(y_1)\right\vert-\left\vert N_\theta(y_2)\right\vert\Big\vert,
\end{align*}
invoking \eqref{eq-mod-Ntheta}, we find
\begin{align*}
    \Big\vert\left\vert N_\theta(y_1)\right\vert-\left\vert N_\theta(y_2)\right\vert\Big\vert
    =&\frac{\left\vert\cos\theta\left(\hat N_\theta(y_2)-\hat N_\theta(y_1)\right)\cdot E_{n+1}\right\vert}{\left(1+\cos\theta\hat N_\theta(y_1)\cdot E_{n+1}\right)\left(1+\cos\theta\hat N_\theta(y_2)\cdot E_{n+1}\right)}\\
    \leq&\frac{\vert\cos\theta\vert}{\left(1-\vert \cos\theta\vert\right)^2}\left\vert\hat N_\theta(y_1)-\hat N_\theta(y_2)\right\vert.
\end{align*}
This, together with the estimate \eqref{ineq-mod-Ntheta} and also \eqref{ineq-Lip-hatN},
gives exactly \eqref{ineq-N-Lip}.

Recall that $\Gamma_{s;\theta}^t$ can be covered by $\{\mcU_j\}$, up to a $\mcH^n$-negligible set, and hence by virtue of Rademacher's theorem, $N_\theta$ is tangentially differentiable along $\Gamma_{s;\theta}^t$ at $\mcH^n$-a.e. $y\in\Gamma_{s;\theta}^t$.

\vspace{0.3cm}
	\noindent{\bf Conclusion of the proof. $C^{1,1}$-rectifiability of $\Gamma_{s;\theta}^+$.} 
	
	By \eqref{eq-DM-3-12}
	and \eqref{ineq-N-Lip}, on each $\mcU_j$, we can use the Whitney-Glaser extension theorem to see that there exists $\phi_\theta\in C^{1,1}(\mfR^{n+1})$ such that $(\de_\theta,N_\theta)=(\phi_\theta,\nabla\phi_\theta)$ on $\mcU_j$.
	Then, by the  ${C^{1,1}}$-Implicit function theorem, for each $y\in \mcU_j$,
	there exists $\psi_j\in C^{1,1}(\left\{\hat N_\theta(y)\right\}^\perp)$ satisfying \eqref{defn-psij-1} and \eqref{defn-psij-2}, which completes the proof.

\end{proof}

\begin{proposition}\label{Prop-DM-3steps}
Given $\theta\in(0,\pi)$, let $\Om$ be a nonempty, bounded, relatively open set with finite perimeter in $\overline{\mfR^{n+1}_+}$, for every $0<t<\infty$, and for a.e. $0<s<t$, 
there holds
	\begin{enumerate}
		\item $N_\theta$ is tangentially differentiable along $\Gamma_{s;\theta}^t$ at $\mcH^n$-a.e. $y\in\Gamma_{s;\theta}^t$, with
		\begin{align}\label{eq-DM-3-15}
		   \begin{cases}
		       \nabla^{\Gamma_{s;\theta}^t}\hat N_\theta(y)=-\sum_{i=1}^n(\kappa_{s;\theta}^t)_i(y)\tau_i(y)\otimes\tau_i(y),\\
		   -1/s\leq(\kappa_{s;\theta}^t)_i(y)\leq1/(t-s),
		   \end{cases}
		\end{align}
			where $\left\{\left( \kappa_{s;\theta}^t\right)_i(y)\right\}_{i=1}^n$ denote the principal curvatures of $\hat N_\theta$ along $\Gamma_{s;\theta}^t$ at $y$ which are indexed in increasing order.
		\item Letting $\Om^\star_\theta:=\bigcup_{s>0}\Gamma_{s;\theta}^+$, then $\vert \Om\Delta\Om^\star_\theta\vert=0.$
	\item 
	For every $r<s<t$, the map $g_{r;\theta}:\Gamma_{s;\theta}^t\ra\Gamma_{s-r;\theta}^t$,
	given by $g_{r;\theta}(y)=y-r(\hat N_\theta(y)+\cos\theta E_{n+1})$ for $y\in\Gamma_{s;\theta}^t$, is a bijection from $\Gamma_{s;\theta}^t$ to $\Gamma_{s-r;\theta}^t$ and is Lipschitz when restricted to each $\mcU_j$, with
	\begin{align}\label{eq-DM-3-39}
		J^{\Gamma_{s;\theta}^t}g_{r;\theta}(y)=\prod_{i=1}^n\left(1+r(\kappa_{s;\theta})_i(y) \right),\quad (\kappa^t_{s-r;\theta})_i(g_{r;\theta}(y))=\frac{(\kappa^t_{s;\theta})_i(y)}{1+r(\kappa^t_{s;\theta})_i(y)},
	\end{align}
for $\mcH^n$-a.e. $y\in\Gamma_s^t$.
	\end{enumerate}
\end{proposition}
\begin{proof}
{\bf (1)}
Consider those $\Gamma_{s;\theta}^t$ resulting from the conclusion of \cref{Thm-C11Rec}.
Recall that we have proved: $N_\theta$ is tangentially differentiable along ${\Gamma_{s;\theta}^t}$ at $\mcH^n$-a.e. $y\in\Gamma_{s;\theta}^t$ in \cref{Thm-C11Rec},
which is done by constructing a sequence of compact sets $\mcU_j$, such that $\mcH^n(\Gamma_s^t\setminus\bigcup_{j=1}^\infty \mcU_j)=0$, where each $\mcU_j$ is proved to be contained in the graph of some $C^{1,1}$-function, on which $N_\theta$ is Lipschitz, see \eqref{ineq-N-Lip}.

By virtue of \cref{DM-2.1(iv)}, to study the tangential gradient of $\hat N_\theta$ along $\Gamma_{s;\theta}^t$, it suffice to work on each $\mcU_j$ (see \eqref{defn-psij-1} and \eqref{defn-psij-2} for the construction of $\mcU_j$).

To proceed, for any $y\in \mcU_j$, we consider a natural Lipschitz extension of $\hat N_\theta$,
from $\mcU_j\cap
\left(y+\mathcal{C}(\hat N_\theta(y),\rho)\right)$ to $y+\mathcal{C}(\hat N_\theta(y),\rho)$, denoted by $\hat N_{\theta\ast}$ and is given by
	\begin{align}
	\hat N_{\theta\ast}(y+z+h\hat N_\theta(y))=\hat N_\theta(y+z),\quad\forall z\in \{\hat N_\theta(y)\}^\perp,\vert z\vert,h
	<\rho_j,
	\end{align}
where $\hat N_\theta(y+z)$ is just the upwards pointing unit normal of the graph $(x',\psi_j(x'))$ at $y+z\in \mcU'_j$.
With such Lipschitz extension and recalling \cref{Prop-Gammast}(3), we can follow the classical argument in \cite[Theorem1, Step1]{DM19} to conclude {\bf (1)}.

\vspace{0.3cm}
\noindent{\bf (2)}
We divide the proof into 2 steps.

\noindent{\bf Step 1. }For every $0<s<t<\infty$, 
there holds
\begin{align}\label{eq-DM-3-34}
    \mcH^n(\pr\Om_t)\leq(t/s)^n\mcH^n(\Gamma_{s;\theta}^t).
\end{align}
Indeed, for $r\in[-s,t-s]$, we consider the map
\begin{align}
    f_{r;\theta}:\Gamma_{s;\theta}^t\ra\pr\Om_{s+r},\quad f_{r;\theta}(y)=y+r(\hat N_\theta(y)+\cos\theta E_{n+1}).
\end{align}
To see that $f_{r;\theta}(y)\in\pr\Om_{s+r}$, we invoke \eqref{eq-DM-3-11}.
Notice that by the definition of $\Gamma_{s;\theta}^t$, the map $f_{t-s;\theta}$ is surjective; that is, $\pr\Om_t=f_{t-s;\theta}(\Gamma_{s;\theta}^t)$. Consequently, we can use the area formula \eqref{formu-area} to see that
\begin{align}
    \mcH^n(\pr\Om_t)
    =\mcH^n(f_{t-s;\theta}(\Gamma_{s;\theta}^t))\leq\int_{f_{t-s;\theta}(\Gamma_{s;\theta}^t)}\mcH^0(f_{t-s;\theta}^{-1}(z))\rd\mcH^n(z)
    =\int_{\Gamma_{s;\theta}^t}{\rm J}^{\Gamma_{s;\theta}^t}f_{t-s}\rd\mcH^n,
\end{align}
where ${\rm J}^{\Gamma_{s;\theta}^t}f_{t-s;\theta}$ denotes the tangential Jacobian of $f_{t-s;\theta}$ along $\Gamma_{s;\theta}^t$. By virtue of {\bf (1)}, a simple computation then yields
\begin{align*}
    {\rm J}^{\Gamma_{s;\theta}^t}f_{t-s;\theta}=\prod_{i=1}^n(1-(t-s)(\kappa_{s;\theta}^t)_i)\leq(1+\frac{t-s}{s})^n=(\frac{t}{s})^n\quad\mcH^n\text{-}a.e.\text{ on }\Gamma_{s;\theta}^t,
\end{align*}
where we have used \eqref{eq-DM-3-15} for the inequality. In particular, this completes the step.

\noindent{\bf Step 2. }$\vert\Om\Delta\Om^\star_\theta\vert=0$.

We first apply the coarea formula to find
\begin{align}\label{ineq-Om-Omstar}
    \frac{\vert\Om\Delta\Om^\star_\theta\vert}{1+\vert\cos\theta\vert}\leq\int_{\Om\Delta\Om^\star_\theta}\vert\nabla \de_\theta\vert \rd\mcL^{n+1}
    =\int_0^\infty\mcH^n\left((\Om\Delta\Om^\star_\theta)\cap\pr \Om_s\right)\rd s
    =\int_0^\infty\mcH^n(\pr\Om_s\setminus\Gamma_{s;\theta}^+)\rd s,
\end{align}
where we have used \eqref{ineq-mod-Ntheta} for the first inequality; the fact that $\Gamma_{s;\theta}^+\subset\pr\Om_s$ for the last equality.

{\bf Claim. }For a.e. $s>0$, there holds
\begin{align}\label{eq-DM-3-38}
    \mcH^n(\Gamma_{s;\theta}^+)=\mcH^n(\pr\Om_s).
\end{align}
If the claim holds, we immediately deduce that $\vert\Om\Delta\Om^\star_\theta\vert=0$ by virtue of \eqref{ineq-Om-Omstar}, which proves {\bf (2)}. Now we prove the claim, using the coarea formula and the estimate \eqref{ineq-mod-Ntheta} again, we find: for every $0< s<t<\infty$,
\begin{align*}
    \int_s^t\mcH^n(\pr\Om_r)\rd r
    =\int_{\Om_s\setminus\Om_t}\vert\nabla \de_\theta\vert \rd\mcL^{n+1}\leq\frac{\vert\Om_s\setminus\Om_t\vert}{1-\vert\cos\theta\vert}
    \leq\frac{\vert\Om\vert}{1-\vert\cos\theta\vert}
    <\infty,
\end{align*}
which implies that the function $\mcH^n(\p\Om_r)$ is integrable on any $[s,t]\subset\mfR^+$. Therefore, we can exploit the Lebesgue-Besicovitch Diﬀerentiation Theorem (see e.g., \cite[Theorem 1.32]{EG15}) to obtain
\begin{align}
    \mcH^n(\pr\Om_s)
    =\lim_{\ep\ra0^+}\frac{1}{2\ep}\int_{-\ep}^\ep\mcH^n(\pr\Om_{s+r})\rd r\quad\text{for a.e. }s>0,
\end{align}
whereby \eqref{eq-DM-3-34} and \cref{Prop-Gammast}{\bf(1)},
\begin{align*}
    \frac{1}{2\ep}\int_{-\ep}^\ep\mcH^n(\pr\Om_{s+r})\rd r
    \leq\frac{1}{2\ep}\int_{-\ep}^\ep(1+r/s)^n\mcH^n(\Gamma_{s;\theta}^{s+r})\rd r
    \leq(1+\ep/s)^n\mcH^n(\Gamma_{s;\theta}^+).
\end{align*}
Since $\Gamma_{s;\theta}^+\subset\pr\Om_s$, this proves
\eqref{eq-DM-3-38} and hence {\bf(2)}.

\vspace{0.3cm}
\noindent{\bf (3)}
By virtue of \eqref{eq-DM-3-11}, $g_{r;\theta}$ is a bijection between $\Gamma_{s;\theta}^t$ and $\Gamma_{s-r;\theta}^t$ for every $r\in(0,s)$ and $t>0$.
We note that if $y$ is a point of tangential differentiability of $\hat N_\theta$ along $\Gamma_{s;\theta}^t$,
then $g_{r;\theta}(y)$ is a point of tangential differentiability of $N_\theta$ along $\Gamma_{s-r;\theta}^t$. Indeed, by virtue of {\bf Claim 1} in the proof of \cref{Lem-u-xi}{\bf(3)}, we have
\begin{align}\label{eq-DM-3-40}
    \hat N_\theta(y)=\hat N_\theta(g_{r;\theta}(y))=\hat N_\theta\left(y-r(\hat N_\theta(y)+\cos\theta E_{n+1})\right)\quad\text{for all }y\in\Gamma_{s;\theta}^t,
\end{align}
so that if $y$ is a point of tangential differentiability of $N_\theta$ along $\Gamma_{s;\theta}^t$ and $\tau\in T_y\Gamma_{s;\theta}^t$, then $\tau\in T_{g_{r;\theta}(y)}\Gamma_{s-r;\theta}^t$ and
\begin{align*}
    (\nabla^{\Gamma_{s;\theta}^t}\hat N_\theta)_y[\tau]=(\nabla^{\Gamma_{s-r;\theta}^t}\hat N_\theta)_{g_{r;\theta}(y)}\left[\tau-r(\nabla^{\Gamma_{s;\theta}^t}\hat N_\theta)_y[\tau]\right].
\end{align*}
Consider the eigenvectors $\{\tau_i(y)\}_{i=1}^n$ of $\hat N_\theta$ in \eqref{eq-DM-3-15}, we find
\begin{align*}
    -(\kappa_{s;\theta}^t)_i(y)\tau_i(y)=(1+r(\kappa_{s;\theta}^t)_i(y))(\nabla^{\Gamma_{s-r;\theta}^t}\hat N_\theta)_{g_{r;\theta}(y)}[\tau_i(y)],
\end{align*}
which implies that $\{\tau_i(y)\}_{i=1}^n$ is also an orthonormal basis for $T_{g_{r;\theta}(y)}\Gamma_{s-r;\theta}^t$, consequently
\begin{align*}
    (\kappa^t_{s-r;\theta})_i(g_{r;\theta}(y))=\frac{(\kappa^t_{s;\theta})_i(y)}{1+r(\kappa^t_{s;\theta})_i(y)}.
\end{align*}
The tangential Jacobian of $g_{r;\theta}$ along $\Gamma_{s;\theta}^t$ can be computed directly, this completes the proof of \eqref{eq-DM-3-39}.
\end{proof}

\section{Proof of the Alexandrov-type theorem}\label{Sec-5}
In this section, we prove \cref{Thm-Alexandrov}.
For simplicity, note that after rescaling, we may assume that $H^0_{\Om;\theta}=n$ in \eqref{defn-H0}.
\begin{proposition}
    Under the assumptions of \cref{Thm-Alexandrov}, there holds that
    \begin{align}\label{eq-DM-3-41}
    \vert\Om^\star_\theta\setminus\zeta_\theta(Z)\vert=0,
\end{align}
where
\begin{align}
    Z=\left\{(x,t)\in{\rm reg}\pr\Om\times\mfR:0<t\leq\frac{1}{\kappa_n(x)}\right\},
\end{align}
and $\zeta_\theta$ is defined through $\zeta_\theta:Z\ra\mfR^{n+1};(x,t)\mapsto x-t(\nu_\Om(x)-\cos\theta E_{n+1})$.
\end{proposition}
\begin{proof}
As mentioned in the introduction, our first concern is the counting measure of $g_{s;\theta}^{-1}(x)$ for $x\in{\rm reg}\pr\Om$.

\noindent{\bf Step 1. }We prove that for all $x\in\pr\Om$, there holds
\begin{align}\label{eq-DM19-3-43}
    \mcH^0(g_{s;\theta}^{-1}(x))\leq{\color {black}2}.
\end{align}
Suppose on the contrary that for some $x\in\pr\Om$, $\mcH^0(g_{s;\theta}^{-1}(x))\geq3$.

If $x$ is in the interior of $\pr\Om$, taking into account that ${\rm var}(\pr\Om)$ is of constant generalized mean curvature, and also the monotonicity formula \cite[Theorem 17.6]{Sim83}, any $C\in{\rm VarTan}({\rm var}(\pr\Om),x)$ would then be a nontrivial stationary varifold with $\Theta^n(\vvert C\vvert,q)\geq1$ for all $q\in{\rm spt}\vvert C\vvert$, whose support is contained in the intersection of two nonopposite half-spaces.
Thanks to the uniform density lower bounds, $C$ is rectifiable by the Rectifiability Theorem \cite[Theorem 42.4]{Sim83}, and it follows from \cite[Theorem 19.3]{Sim83} that $C$ is a rectifiable cone. By construction, $\vvert C\vvert$ is supported in two non-opposite half-spaces, and hence we can use the interior strong maximum principle \cref{Lem-MP1} to derive a contradiction.

If $x$ is a boundary point of $\pr\Om$, namely, $x\in\Gamma\subset\p\mfR^{n+1}_+$,
we have the following crucial observation.

\footnote{If $x$ is a regular point, the claim has been proved in \cite[Proof of Theorem 1.1, Case 2]{JWXZ22-B}.}{\bf Claim. }For all $x\in\Gamma$, it can not be that $x=g_{s;\theta}(y)$ for some $y\in\Om\cap\mfR^{n+1}_+$ (namely, $y$ can not be those points in $\Om$ that lie in the open upper half-space).

Suppose on the contrary that there exists some $s>0$ and some $y\in\Om\cap\mfR^{n+1}_+$ such that $x=g_{s;\theta}(y)$,
our aim is to show that this case is not possible.
Indeed, let us set $\nu:=\frac{y-s\cos\theta E_{n+1}-x}{\vert y-s\cos\theta E_{n+1}-x\vert}$, since $y$ lies in the interior of $\Om$, and recall that $\de_\theta(y)=s$, we know that $$\nu\cdot(-E_{n+1})
=-\frac{y\cdot E_{n+1}}{s}+\cos\theta<\cos\theta,$$
which means that $\alpha:=\arccos{(\nu\cdot-E_{n+1})}>\theta$ strictly.

Our first observation is that the projection of $y$ onto $\p\mfR^{n+1}_+$ must be in $T$, otherwise, the point that $y$ attains its shifted distance can not be in $\Gamma$, which contradicts to $x=g_{s;\theta}(y)$.

Then we observe that, thanks to that $\Gamma$ is a smooth $(n-1)$-manifold in $\p\mfR^{n+1}$, the blow-up limit $\tilde T$ is a $n$-half plane in $\p\mfR^{n+1}_+$ and $\tilde\Gamma$ is the relative boundary of $\tilde T$ in $\p\mfR^{n+1}_+$ (and is of course an $(n-1)$-plane).
Since $x=g_{s;\theta}(y)$ and $H^-=\{z\cdot\nu\leq0\}$, we know that the $(n-1)$-plane $\tilde\Gamma$ must coincide with the $(n-1)$-plane $\p H^-\cap\p\mfR^{n+1}_+$, otherwise there exists some $x'\neq x\in\Gamma$ such that ${\rm dist}(y-s\cos\theta E_{n+1},x')<s={\rm dist}(y-s\cos\theta E_{n+1},x)$, which contradicts again to the fact that $x=g_{s;\theta}(y)$.

From the second observation, we know that $\tilde T$ must be either $H^+\cap\p\mfR^{n+1}_+$ or $H^-\cap\p\mfR^{n+1}$, and the possibility that $\tilde T=H^-\cap\p\mfR^{n+1}_+$ is ruled out by virtue of the first observation.
By virtue of the arguments in \cref{Sec-2-9},
there exists a blow-up limit at $x$, which is a triple of rectifiable cones, say $\left({\rm var}(\tilde M,\psi_1),{\rm var}(\tilde T),{\rm var}(\tilde\Gamma)\right)$, that is $\theta$-stationary.
Moreover, by construction, the support of  ${\rm var}(\tilde M,\psi_1)$ is contained in the closed half-space $H^-=\{z\in\mfR^{n+1}:z\cdot\nu\leq0\}$.
Therefore, we can use the boundary maximum principle \cref{Cor-MP3} for the
$\theta$-stationary triple $\left({\rm var}(\tilde M,\psi_1),{\rm var}(\tilde T),{\rm var}(\tilde\Gamma)\right)$ to derive a contradiction to the fact that $\pi>\alpha>\theta>0$, which proves the claim.

From the claim we see that when $x\in\Gamma$, it must be that $x=g_{s;\theta}(y)$ for some $y\in\overline\Om\cap\p\mfR^{n+1}_+\subset T$.
In this case, following the proof of the claim, and notice that $\alpha=\theta$ since $y\in\p\mfR^{n+1}_+$, we can use again the boundary maximum principles \cref{Lem-MP3}, \cref{Cor-MP3} to see that when $\theta\in(0,\pi/2)\cup(\pi/2,\pi)$,
$H^-$ is uniquely determined by the point $y\in T$ such that $x=g_{s;\theta}(y)$,
so that there exists at most one $y\in T$ such that $x=g_{s;\theta}(y)$.
\eqref{eq-DM19-3-43} is thus completed.

\

\noindent{\bf Step 2. }We prove \eqref{eq-DM-3-41}.

Using the coarea formula, we find
\begin{align*}
    \frac{\vert\Om^\star_\theta\setminus\zeta_\theta(Z)\vert}{1+\vert\cos\theta\vert}
    \leq\int_{\Om^\star_\theta\setminus\zeta_\theta(Z)}\vert\nabla\de_\theta\vert \rd\mcL^{n+1}
    =&\int_0^\infty\mcH^n\left((\Om^\star_\theta\setminus\zeta_\theta(Z))\cap\Gamma_{s;\theta}^+\right)\rd s
    =\int_0^\infty\mcH^n\left(\Gamma_{s;\theta}^+\setminus\zeta_\theta(Z)\right)\rd s,
\end{align*}
here we used the estimate \eqref{ineq-mod-Ntheta} for the inequality and \eqref{eq-DM-3-38} for the last equality.

Since $x\in{\rm reg}\pr\Om$ and $y\in\Gamma_{s;\theta}^+$ are such that $y=x-s(\nu_\Om(x)-\cos\theta E_{n+1})$ if and only if $x=y-s(\hat N_\theta(y)+\cos\theta E_{n+1})=g_{s;\theta}(y)$, with $g_{s;\theta}$ as in \cref{Prop-DM-3steps}{\bf(3)}.
In particular, we have\footnote{Here we make the following remarks: {\bf 1.} Recalling the definition of $\Gamma_{s;\theta}^t$, we know that for any $y\in\Gamma_{s;\theta}^+$ such that $B_s(y-\cos\theta E_{n+1})$ touches $\pr\Om$ from the interior at some $x\in{\rm reg}(\pr\Om)$, there holds $1/\kappa_n(x)\geq s$, see \cite[Proof of Theorem 1.1]{JWXZ22-B}. Therefore, from the `if and only if' observation above, we know that $y\in\zeta_\theta(Z)$; {\bf 2.} By $g_{s;\theta}^{-1}$ we mean, the pre-image of those $x\in\pr\Om$ that are mapped from $\Gamma_{s;\theta}^+$ through the map $g_{s;\theta}$.}
\begin{align*}
    \zeta_\theta(Z)\cap\Gamma_{s;\theta}^+=g_{s;\theta}^{-1}\left({\rm reg}\pr\Om\right)\quad\text{for all }s>0.
\end{align*}
Taking into account that $g_{s;\theta}^{-1}(\pr\Om)\subset\Gamma_{s;\theta}^+$, in order to prove 
\eqref{eq-DM-3-41}, we are left to show that for a.e. $s>0$, there holds
\begin{align}\label{eq-DM-3-42}
    \mcH^n\left(g_{s;\theta}^{-1}\left({\rm sing}\pr\Om\right)\right)=0.
\end{align}
In other words, the points in $\Gamma_{s;\theta}^+$ that projected over $\pr\Om$, end up on the singular set ${\rm sing}\pr\Om$, have negligible $\mcH^n$-measure. Indeed, we will show that \eqref{eq-DM-3-42} holds for every $s>0$ such that \eqref{eq-DM-3-38} holds, and we shall proceed the proof by a contradiction argument. Precisely, assuming that 
$\mcH^n(\Gamma_{s;\theta}^+)=\mcH^n(\p\Om_s)$ with
\begin{align}\label{ineq-DM-3-42-2}
    \mcH^n\left(g_{s;\theta}^{-1}\left({\rm sing}\pr\Om\right)\right)>0.
\end{align}
In particular, there exists some $t>s$, such that $\mcH^n\left(\Gamma_{s;\theta}^t\cap g_{s;\theta}^{-1}\left({\rm sing}\pr\Om\right)\right)>0$.

Now we invoke the key assumption $\mcH^{n-1}\left({\rm sing}\pr\Om\right)=0$, using the area formula \eqref{formu-area} and \eqref{eq-DM19-3-43}, we find
\begin{align*}
    0=2\mcH^n\left({\rm sing}\pr\Om\right)
    \geq\int_{{\rm sing}\pr\Om}\mcH^0(g_{s;\theta}^{-1}(x))\rd\mcH^n(x)
    =\int_{g_{s;\theta}^{-1}\left({\rm sing}\pr\Om\right)}{\rm J}^{\Gamma_{s;\theta}^t}g_{s;\theta}\rd\mcH^n,
\end{align*}
where ${\rm J}^{\Gamma_{s;\theta}^t}g_{s;\theta}=\prod_{i=1}^n(1+s(\kappa_{s;\theta}^t)_i)\geq0$ along $\Gamma_{s;\theta}^t$ thanks to \cref{Prop-DM-3steps}{\bf(1)} and \eqref{eq-DM-3-15}. Having assumed \eqref{ineq-DM-3-42-2}, and since $\{(\kappa_{s;\theta}^t)_i\}_{i=1}^n$ are indexed in the increasing order, we deduce that
\begin{align}\label{ineq-DM-3-44}
    \mcH^n\left(\left\{y\in\Gamma_{s;\theta}^t:(\kappa_{s;\theta}^t)_1(y)=-1/s\right\}\right)\geq\mcH^n\left(\Gamma_{s;\theta}^t\cap g_{s;\theta}^{-1}\left({\rm sing}\pr\Om\right)\right)>0.
\end{align}
By \eqref{eq-DM-3-39}, we have
\begin{align*}
    \La^t_{s-r}:=
    \left\{\tilde y\in\Gamma_{s-r;\theta}^t:(\kappa_{s-r;\theta}^t)_1(\tilde y)=-1/(s-r)\right\}
    =g_{r;\theta}\left(\left\{y\in\Gamma_{s;\theta}^t:(\kappa_{s;\theta}^t)_1(y)=-1/s\right\}\right).
\end{align*}
Recall that $g_{r;\theta}:\Gamma_{s;\theta}^t\ra\Gamma_{s-r;\theta}^t$ is injective, the area formula \eqref{formu-area} then yields
\begin{align*}
    \mcH^n\left(\La^t_{s-r}\right)=\int_{\left\{y\in\Gamma_{s;\theta}^t:(\kappa_{s;\theta}^t)_1(y)=-1/s\right\}}{\rm J}^{\Gamma_{s;\theta}^t}g_{r;\theta}\rd\mcH^n.
\end{align*}
Using again \eqref{eq-DM-3-15}, we have: for every $r\in(0,s)$, ${\rm J}^{\Gamma_{s;\theta}^t}g_{r;\theta}=\prod_{i=1}^n(1+r(\kappa_{s;\theta}^t)_i)\geq(1-r/s)^n>0$ along $\Gamma_{s;\theta}^t$, and hence \eqref{ineq-DM-3-44} implies that: for every $r\in(0,s)$, there holds
\begin{align}\label{ineq-DM-3-45}
    \mcH^n(\La^t_{s-r;\theta})>0.
\end{align}
Notice that the map $a\mapsto a/(1+ra)$ is increasing on $a\geq0$, using \eqref{eq-DM-3-39}, we find: for every $\tilde y\in\La_{s-r;\theta}^t$ with $\tilde y=g_{r;\theta}(y)$ for some $y\in\Gamma_{s;\theta}^t$, there holds
\begin{align}\label{ineq-DM-3-46}
    \sum_{i=1}^n(\kappa_{s-r;\theta}^t)_i(\tilde y)=-\frac{1}{s-r}+\sum_{i=2}^n\frac{(\kappa_{s;\theta}^t)_i(y)}{1+r(\kappa_{s;\theta}^t)_i(y)}
    \leq-\frac{1}{s-r}+(n-1)\frac{1/(t-s)}{1+(r/(t-s))}
    \leq0,
\end{align}
provided $r\in(r_0,s)$ for some $r_0$ depending on $s,t$ that is close enough to $s$.

Let us consider the set
\begin{align*}
    \La_\theta:=\bigcup_{r_0<r<s}\La_{s-r;\theta}^t.
\end{align*}
By the coarea formula, the estimate \eqref{ineq-mod-Ntheta} and \eqref{ineq-DM-3-45}, we have
\begin{align*}
    \frac{\vert\La_\theta\vert}{1-\vert\cos\theta\vert}
    \geq\int_{\La_\theta}\vert\nabla\de_\theta\vert \rd\mcL^{n+1}
    =\int_{r_0}^s\mcH^n(\La_\theta\cap\p\Om_{s-r})\rd r
    =\int_{r_0}^s\mcH^n(\La_{s-r;\theta}^t)\rd r>0.
\end{align*}

{\color{black}
Recall that we have used the Implicit function theorem to obtain $C^{1,1}$-functions $\psi_j$ so that each $\mcU_j$ is indeed the graph of $\psi_j$ (see \eqref{defn-psij-2}).
By virtue of the fact that $\psi_j$ admits second-order differential $\nabla^2\psi_j$ for $\mcH^n$-a.e. points of $\mcU_j$,
\eqref{ineq-DM-3-45},
and the fact that $\vert\La_\theta\vert>0$, we have:
there exists some $r\in(r_0,s)$ and some $y_0\in\La_{s-r;\theta}^t\cap\mfR^{n+1}_+\subset\Gamma_{s-r;\theta}^t\cap\mfR^{n+1}_+$, at which $\psi_j$ is second-order differentiable, such that $\hat N_\theta$ is tangential differentiable along $\Gamma_{s-r;\theta}^t$ at $y_0$, with $x_0:=\xi_\theta(y_0)$ lying in the interior of $\pr\Om$ (thanks to the claim in the proof of \eqref{eq-DM19-3-43}), and of course
\begin{align}\label{eq-DM-3-47}
    \nabla^{\Gamma_{s-r;\theta}^t}\hat N_\theta(y_0)
    =\left(-\sum_{i=1}^n(\kappa_{s-r;\theta}^t)_i(y_0)\right)\tau_i(y_0)\otimes\tau_i(y_0)
\end{align}
thanks to \eqref{eq-DM-3-15}.
Moreover, by \eqref{ineq-DM-3-46}, we also have
\begin{align}\label{ineq-DM-3-48}
    \sum_{i=1}^n(\kappa_{s-r;\theta}^t)_i(y_0)\leq0.
\end{align}
Now we set $\nu=-\hat N_\theta(y_0)$ and
\begin{align*}
    \mathbf{D}_\rho:=\{z\in\nu^\perp: \vert z\vert<\rho\},\quad\mathbf{C}_\rho:=\{z+h\nu:z\in\mathbf{D}_\rho,\vert h\vert<\rho\},\quad\text{for }\rho>0.
\end{align*}
By virtue of the second-order differentiability of $\psi_j$ at $y_0$, \eqref{defn-psij-2}, and notice that for the level-set $\pr\Om_s=\{\de_\theta=s\}$, $\nu$ is its unit normal at $y_0$, and that the mean curvature at $y_0$ is given by
$$\vert\nabla\de_\theta(y_0)\vert H_{\pr\Om_s}(y_0)=\De\de_\theta(y_0)-\nabla^2\de_\theta(\nu,\nu)=\sum_{i=1}^n\nabla^2\de_\theta(\tau_i(y_0),\tau_i(y_0)),$$
taking also \eqref{eq-DM-3-47}
and \eqref{ineq-DM-3-48} into account, for every $\ep>0$, we have: there exists some $\rho>0$ and a second-order polynomial $\eta:\nu^\perp\cong\mfR^n\ra\mfR$, such that $\eta(0)=0,\nabla\eta(0)=0$, and
\begin{align}
    -{\rm div}\left(\frac{\nabla\eta}{\sqrt{1+\vert\nabla\eta\vert^2}}\right)(z)
    \leq&-{\rm div}\left(\frac{\nabla\eta}{\sqrt{1+\vert\nabla\eta\vert^2}}\right)(0)+\ep\notag\\
    \leq&\sum_{i=1}^n(\kappa_{s-r;\theta}^t)_i(y_0)+2\ep
    \leq2\ep
\end{align}
for every $z\in\mathbf{D}_\rho$ and
\begin{align}
    y_0+\{z+h\nu:z\in\mathbf{D}_\rho,-\rho<h<\eta(z)\}\subset(y_0+\mathbf{C}_\rho)\cap\Om_{s-r}.
\end{align}
Now we translate $\Om$ by $(s-r)(\hat N_\theta(y_0)+\cos\theta E_{n+1})$, recall that $x_0=\xi_\theta(y_0)$ lies in the interior of $\pr\Om$,
we thus find: there exists a small enough $\rho_1>0$, such that
\begin{align*}
    \Om_{s-r}\cap(y_0+\mathbf{C}_{\rho_1})\subset(\Om+(s-r)(\hat N_\theta(y_0)+\cos\theta E_{n+1}))\cap(y_0+\mathbf{C}_{\rho_1}),
\end{align*}
with
\begin{align*}
    y_0\in\pr\Om_{s-r}\cap\p\left(\Om+(s-r)(\hat N_\theta(y_0)+\cos\theta E_{n+1})\right)\cap(y_0+\mathbf{C}_{\rho_1}).
\end{align*}
We are in the position to exploit the Sch\"atzle's
strong maximum principle \cref{Thm-SMP}, with
\begin{align*}
    M=\p\left(\Om+(s-r)(\hat N_\theta(y_0)+\cos\theta E_{n+1})\right),
\end{align*}
$\nu=-\hat N(y_0)$, $U=\mathbf{D}_{\rho_1}$, $z_0=0$, $h_0=\nu\cdot y_0-\rho_1$ and $\eta$ as above. Notice that if we set
\begin{align*}
    \varphi(z)=\inf\{h\in(h_0,\infty):z+h\nu\in M\},\quad z\in\mathbf{D}_{\rho_1},
\end{align*}
then we have $h_0<\eta\leq\varphi<\infty$ on $\mathbf{D}_{\rho_1}$, and of course $\varphi(0)=\eta(0)=0$. Arguing as in \cite[Step 4]{DM19}, by comparing the mean curvatures, we deduce a contradiction due to the Sch\"atzle's strong maximum principle \cref{Thm-SMP}. In particular, this proves \eqref{eq-DM-3-41}.
}

\end{proof}

\begin{proof}[Proof of \cref{Thm-Alexandrov}]
By virtue of \eqref{eq-DM-3-41}, we are now prepared to apply the classical flow argument in \cite{JWXZ22-B}.

\noindent{\bf Step 1. }Applying the flow method.

By virtue of \cref{Prop-DM-3steps}{\bf(2)} and \eqref{eq-DM-3-41}, using the area formula, we have
\begin{align*}
    \vert\Om\vert
    =&\vert\Om^\star_\theta\vert
    \leq\vert\zeta_\theta(Z)\vert
    \leq\int_Z\mcH^0(\zeta_\theta^{-1}(y))\rd\mcL^{n+1}(y)
    =\int_Z{\rm J}^Z\zeta_\theta \rd\mcL^{n+1},
\end{align*}
a classical computation gives
\begin{align*}
    {\rm J}^Z\zeta_\theta(x,t)=\left(1-\cos\theta\nu_{\Om}(x)\cdot E_{n+1}\right)\prod_{i=1}^n(1-t\kappa_i(x)),
\end{align*}
and hence we have
\begin{align*}
\vert\Om\vert
\leq\int_Z{\rm J}^Z\zeta_\theta \rd\mcL^{n+1}
=\int_{{\rm reg}\pr\Om} \rd\mcH^n(x)\int_0^{\frac{1}{\max\left\{\kappa_i(x)\right\}}}\left(1-\cos\theta\nu_{\Om}\cdot E_{n+1}\right)\prod_{i=1}^n(1-t\kappa_i(x))\rd t.
\end{align*}
By the AM-GM inequality, 
and the fact that $\max\left\{\kappa_i(x)\right\}\geq H^0_{\Om;\theta}/n$, we find
\begin{align*}
    \vert\Om\vert
    \leq&\int_{{\rm reg}\pr\Om} \rd\mcH^n(x)\int_0^{\frac{1}{\max\left\{\kappa_i(x)\right\}}} \left(1-\cos\theta\nu_{\Om}\cdot E_{n+1}\right)\left(\frac{1}{n}\sum_{i=1}^n\left(1-t\kappa_i(x)\right)\right)^n \rd t\notag\\
    \leq&\int_{{\rm reg}\pr\Om} \left(1-\cos\theta\nu_{\Om}\cdot E_{n+1}\right)\rd\mcH^n\int_0^{\frac{n}{H^0_{\Om;\theta}}} \left(1-t\frac{H^0_{\Om;\theta}}{n}\right)^n \rd t\notag\\
    =&\frac{n}{n+1}\int_{{\rm reg}\pr\Om} \frac{\left(1-\cos\theta\nu_{\Om}\cdot E_{n+1}\right)}{H^0_{\Om;\theta}} \rd\mcH^n(x).
\end{align*}
Invoking the definition of $H^0_{\Om;\theta}$ in \eqref{defn-H0}, we see that equalities hold throughout the argument. In particular, we have
\begin{align}
    \vert\zeta_\theta(Z)\setminus\Om\vert=0,\label{eq-DM-3-52}\\
    \mcH^0(\zeta_\theta^{-1}(y))=1\quad\text{for a.e. }y\in\Om,\label{eq-DM-3-53}\\
    \kappa_{i}(x)=\frac{H^0_{\Om;\theta}}{n}\quad\text{for every }x\in{\rm reg}\pr\Om,\text{ i=1,}\ldots,n.\label{eq-DM-3-54}
\end{align}
\noindent{\bf Step 2. }Analysis of singularities.

We have shown that the regular part of $\p\Om$ must be spherical.
Now we are going to analyze the singularities by virtue of, again, the Sch\"atzle's strong maximum principle.

Indeed, recall that we have rescaled $\Om$ so that $H^0_{\Om;\theta}=n$. By virtue of \eqref{eq-DM-3-54} and the Young's law,
since ${\rm reg}\pr\Om$ is relatively open in $\pr\Om$, we can find a family $\{S_{\theta,i}\}_{i\in I}$, $I\subset\mathbb{N}$, of mutually disjoint subsets of ${\rm reg}\pr\Om$, with $S_{\theta,i}\subset \p B_1(x_i-\cos\theta E_{n+1})\cap\overline{\mfR^{n+1}_+}$ for points $x_i\in\p\mfR^{n+1}_+$ or $x_i\in\mfR^{n+1}_+$ with $\left(x_i-\cos\theta E_{n+1}\right)\cdot E_{n+1}\geq1$, where $\p B_1(x_i-\cos\theta E_{n+1})\cap\overline{\mfR^{n+1}_+}$ is either a $\theta$-cap or a sphere that lies completely in $\overline{\mfR^{n+1}_+}$, such that
\begin{align}\label{eq-DM-3-55}
    {\rm reg}\pr\Om
    =\bigcup_{i\in I}S_{\theta,i},\quad S_{\theta,i}\text{ is relatively open in }\pr\Om\text{ and is connected}.
\end{align}
Since $S_{i;\theta}\subset\pr\Om$, we know that $\de_\theta(x_i)\leq1.$

{\bf Claim. }For every $i\in I$, $\de_\theta(x_i)=1$.

Suppose not, then there exists some constant $\delta>0$, such that $\de_\theta(x_i)=1-4\delta$. Notice that $(B_{(1-\vert\cos\theta\vert)\delta}(x_i)\cap\mfR^{n+1}_+)\cap A_i\subset\Om$ is nonempty, where $A_i=\zeta_\theta(S_i\times(0,1))$ is an open subset of $\Om$.
For any $y\in B_{(1-\vert\cos\theta\vert)\delta}(x_i)\cap\mfR^{n+1}_+\cap A_i$, we have the following observations:
\begin{align*}
    &{\rm dist}\left(y-(1-3\delta)\cos\theta E_{n+1},x_i-(1-4\delta)\cos\theta E_{n+1}\right)\\
    =&\vert y-x_i-\delta\cos\theta E_{n+1}\vert
    \leq\vert y-x_i\vert+\delta\vert\cos\theta\vert
    <\delta,
\end{align*}
where we used the fact that $y\in B_{(1-\vert\cos\theta\vert)\delta}(x_i)$ for the last inequality. Using the triangle inequality again, we have
\begin{align*}
    &{\rm dist}\left(y-(1-3\delta)\cos\theta E_{n+1},\pr\Om\right)
    \leq{\rm dist}\left(y-(1-3\delta)\cos\theta E_{n+1},\xi_\theta(x_i)\right)\\
    \leq&{\rm dist}\left(y-(1-3\delta)\cos\theta E_{n+1},x_i-(1-4\delta)\cos\theta E_{n+1}\right)+{\rm dist}\left(x_i-(1-4\delta)\cos\theta E_{n+1},\xi_\theta(x_i)\right)\\
    <&\delta+1-4\delta=1-3\delta,
\end{align*}
due to \eqref{eq-dist-theta} and \eqref{ineq-weighted-distance}, 
it can not be that $\de_\theta(y)\geq1-3\de$,
which implies $\de_\theta(y)<1-3\delta$. On the other hand, for any $x\in S_{\theta,i}\subset\p B_1(x_i-\cos\theta E_{n+1})\cap\mfR^{n+1}_+$, there holds
\begin{align*}
    &{\rm dist}\left(y-(1-\delta)\cos\theta E_{n+1},x\right)=\vert y-(1-\delta)\cos\theta E_{n+1}-x\vert\notag\\
    =&\vert\left( y-(1-\delta)\cos\theta E_{n+1}-(x_i-\cos\theta E_{n+1})\right)+(x_i-\cos\theta E_{n+1}-x)\vert\notag\\
     \geq&\vert x_i-\cos\theta E_{n+1}-x\vert-\vert y-x_i+\delta\cos\theta E_{n+1}\vert\notag\\
     =&1-\vert y-x_i+\delta\cos\theta E_{n+1}\vert
     \geq1-\left(\vert y-x_i\vert+\delta\vert\cos\theta\vert\right)\notag\\
     >&1-\delta\vert\cos\theta\vert-\delta(1-\vert\cos\theta\vert)=1-\delta,
\end{align*}
where we have used the triangle inequality for the first and the second inequality; the fact that $y\in B_{(1-\vert\cos\theta\vert)\delta}(x_i)$ for the last inequality. In view of \eqref{ineq-weighted-distance}, we know that, when restricted to $S_{\theta,i}$, $\de_\theta(y)>1-\delta$.
In particular, combining above observations, we have: for any $y\in B_{(1-\vert\cos\theta\vert)\delta}(x_i)\cap\mfR^{n+1}_+\cap A_i$, it must be that $\xi_\theta(y)\notin S_{\theta,i}$.
Since \cref{Prop-DM-3steps}{\bf(2)} and \eqref{eq-DM-3-41} imply that for a.e. $y\in\Om$, there exists $x\in{\rm reg}\pr\Om$, such that $\xi_\theta(y)=x$, we conclude from \eqref{eq-DM-3-55} that for a.e. $y\in B_{(1-\vert\cos\theta\vert)\delta}(x_i)\cap\mfR^{n+1}_+\cap A_i$, there exists $j\neq i\in I$ and $x\in S_j$ such that $\xi_\theta(y)=x$. In particular, $B_{(1-\vert\cos\theta\vert)\delta}(x_i)\cap\mfR^{n+1}_+\cap A_i\cap A_j$ is a nonempty, relatively open set in $\mfR^{n+1}_+$, and hence we definitely have
\begin{align*}
    0<\vert B_{(1-\vert\cos\theta\vert)\delta}(x_i)\cap\mfR^{n+1}_+\cap A_i\cap A_j\vert,
\end{align*}
where $A_i\cap A_j\subset\{y\in\Om:\mcH^0(\xi_\theta^{-1}(y))\geq2\}$.
However, this contradicts to \eqref{eq-DM-3-53} and proves the claim.

To complete the proof, we are left to show that for each $i$, the closure of $S_i$, say $T_i$, is either a complete $\theta$-cap or a complete sphere, namely, $T_i=\p B_1(x_i-\cos\theta E_{n+1})\cap\overline{\mfR^{n+1}_+}$. To this end, since $\de_\theta(x_i)=1$ for every $i\in I$\footnote{By definition, we have: ${\rm dist}(x_i-\cos\theta E_{n+1},\pr\Om)=1$, in other words, the part of $\pr\Om$ that does not belong to $T_i$, if exists, must be lying outside $\overline B_1(x_i-\cos\theta E_{n+1})\cap\overline{\mfR^{n+1}_+}$.},
we can apply the strong maximum principle \cref{Thm-SMP} with $M=\pr\Om$ at each $x\in T_i$ that lies in $\mfR^{n+1}_+$, comparing with 
$\p B_1(x_i-\cos\theta E_{n+1})$ locally near $x$, to find some $\rho_x>0$, such that
\begin{align*}
    \pr\Om\cap B_{\rho_x}(x)\cap \p B_1(x_i-\cos\theta E_{n+1}) 
    =\p B_1(x_i-\cos\theta E_{n+1})\cap B_{\rho_x}(x).
\end{align*}
This in particular implies that in the open half-space, we have $\p B_1(x_i-\cos\theta E_{n+1})\cap\mfR^{n+1}_+\subseteq\pr\Om\cap\mfR^{n+1}_+$,
and hence the whole $\theta$-cap $\p B_1(x_i-\cos\theta E_{n+1})\cap\overline{\mfR^{n+1}_+}\subseteq\pr\Om.$

Finally, the fact that $\Om$ is a set of finite perimeter implies that $I\subset\mathbb{N}$ is indeed finite, and it follows that $\pr\Om$ is the union of finitely many  $\theta$-caps and spheres with equal radii, which completes the proof.
\end{proof}

\appendix
\section{Cut-off functions near singularities}\label{App-1}
We begin by constructing useful cut-off functions near ${\rm sing}\pr\Om$.
The technic is standard 
and these cut-off functions are very useful for the study of surfaces with singularities, see e.g., \cite{SS81,Ilmanen96,Wickramasekera14,DePM17,Zhu18}.

\begin{lemma}\label{Lem-Cut-off}
	Let $\Om$ be a non-empty, bounded, relatively open set with finite perimeter in $\bar\mfR^{n+1}_+$ such that $\mcH^{n-1}({\rm sing}\pr\Om)=0$.
	If $\pr\Om$ has Euclidean volume growth,
	then for any $\ep>0$, there exist open sets $S_\ep'\subset S_\ep\subset\mfR^{n+1}$ with ${\rm sing}\pr\Om\subset S_\ep'$ and $S_\ep\subset\{x:{\rm dist}(x,{\rm sing}\pr\Om)<\ep\}$, and there exists a smooth cut-off function $\varphi_\ep\in C^\infty(\mfR^{n+1})$ such that $0\leq\varphi_\ep(x)\leq1$ with
	\begin{align}\label{varphi-construction}
		\varphi_\ep(x)=
		\begin{cases}
			0\quad &x\in S_\ep',\\
			1\quad &x\in\mfR^{n+1}\setminus S_\ep.
		\end{cases}
	\end{align}
	Moreover, $\varphi_\epsilon$ satisfies the following properties:
	\begin{align}\label{pw-conv}
		\varphi_\ep(x)\to 1 \hbox{ pointwisely for }x\in {\rm reg}\pr\Om,
	\end{align}
	\begin{align}\label{esti-cutoff-1}
		\int_{\pr\Om}\vert\nabla^{\pr\Om}\varphi_\epsilon(x)\vert \rd\mcH^{n}(x)\leq C\ep,
	\end{align}
	where $\nabla^{\pr\Om}$ denotes the tangential gradient of $\varphi_\ep$ with respect to the tangent space $T_x\p\Om$, which exists for $\mcH^n$-a.e. along $\pr\Om$ thanks to the condition that $\mcH^{n-1}({\rm sing}\pr\Om)=0$.
	Here and in all follows, $C$ will be referred to as positive constants that are independent of $\ep$.
\end{lemma}
\begin{proof}
	We begin by noticing that ${\rm sing}\pr\Om$ is compact since it is closed and bounded.
	
	For any $\epsilon>0$, since $\mcH^{n-1}({\rm sing}\pr\Om)=0$, we may cover the singular set ${\rm sing}\pr\Om$ with finitely many balls $\mcG:=\{B_{r_i}(z_i)\}_{i=1}^{N_1}$ where $z_i\in \pr\Om$, $\sum_{i=1}^{N_1} r_i^{n-1}<\epsilon$,
	and we may assume without loss of generality that $r_i<{\rm diam}(\Om)$ for each $i$,
    so that the Euclidean volume growth condition \eqref{ineq-vg-M} is valid for $z_i$, see the proof of \cref{Prop-Evg}.
	
	For each $i$, let $\varphi_i\in C^\infty(\mfR^{n+1})$ satisfy $0\leq\varphi_i\leq1$ with
	\begin{align*}
		\varphi_i(x)
		=\begin{cases}
			0\quad &\forall x\in B_{r_i}(z_i),\\
			1\quad &\forall x\in \mfR^{n+1}\setminus B_{2r_{i}}(z_i),
		\end{cases}
	\end{align*}
	and $\vert \nabla\varphi_i(x)\vert\leq\frac{2}{r_i}$ for all $x\in\mfR^{n+1}$.
	
	Define $\tilde\varphi_\epsilon$ by
	\begin{align*}
		\tilde{\varphi}_\epsilon(x):=\min_i \varphi_i(x).
	\end{align*}
	It follows that $\tilde\varphi_\epsilon$ is piecewise-smooth with $0\leq\tilde{\varphi}_\epsilon\leq1$, and
	\begin{align}\label{varphitilde}
		\tilde{\varphi}_\epsilon(x)=
		\begin{cases}
			0 \quad \text{on}\quad  \bigcup_i B_{r_i}(z_i)\supseteq {\rm sing}\pr\Om,\\
			1 \quad \text{on}\quad \mfR^{n+1}\setminus \bigcup_iB_{2r_i}(z_i).
		\end{cases}
	\end{align}
	
	By the Euclidean volume growth condition of $\pr\Om$ and that $\sum_{i=1}^{N_1} r_i^{n-1}<\epsilon$, we have
	\begin{align}\label{ineq-cutoff-1}
		\int_{\pr\Om}\vert\nabla^{\pr\Om}\tilde{\varphi}_\epsilon(x)\vert\rd\mcH^{n}(x)
		&\leq\sum_{i=1}^{N_1}\int_{\pr\Om\cap\left(B_{2r_i}(z_i)\setminus B_{r_i}(z_i)\right)}\vert\nabla\varphi_i(x)\vert\rd\mcH^{n}(x)\notag\\
		&\leq \sum_{i=1}^{N_1} \frac{2}{r_i}\mcH^{n}\left(\pr\Om\cap B_{2r_i}(z_i)\right)\notag\\
		&\leq 2^{n+1}C_1\sum_i^{N_1} r_i^{n-1}\leq2^{n+1}C_1\epsilon.
	\end{align}
	
	We mollify $\tilde{\varphi}_\epsilon$ to obtain a smooth function $\varphi_\epsilon$, which still satisfies estimate of the form \eqref{ineq-cutoff-1} with some constant $C$ that is independent of the choice of $\ep$.
	Since $\tilde{\varphi}_\epsilon$ satisfies \eqref{varphitilde}, we may let $S_\epsilon',S_\epsilon$ denote the sets such that
	\begin{align*}
		\varphi_\epsilon(x)=
		\begin{cases}
			0\quad &x\in S_\epsilon',\\
			1\quad &x\in\mfR^{n+1}\setminus S_\epsilon.
		\end{cases}
	\end{align*}
	It is clear that \eqref{varphi-construction} holds and hence \eqref{pw-conv} is true.	
	We see that $\varphi_\epsilon$ is the desired smooth cut-off function, and this completes the proof.
\end{proof}
Using these cut-off functions, we can prove the tangential divergence theorem for $\mcA$-stationary set.
\begin{proof}[Proof of \cref{Lem-divergenctheorem}]
	We first observe that by virtue of
	\eqref{formu-LZZ21-1.7}, the $\mcA$-stationarity of $\Om$, and \cref{Rem-integrals}, we have: for any $X\in\mathfrak{X}_c(\overline{\mfR^{n+1}_+})$, there holds
	\begin{align*}
		\int_{\pr\Om}{\rm div}_{\p\Om}X\rd\mcH^n
		=\int_{\pr\Om}X\cdot(c\nu_{\Om})\rd\mcH^n.
	\end{align*}
	Since $\mcH^{n-1}({\rm sing}\pr\Om)=0$, we may apply the Allard's regularity theorem \cite[Theorem 24.2]{Sim83} to ${\rm var}(\pr\Om)$, and see that $\pr\Om\cap\mfR^{n+1}_+$ is an analytic hypersurface with constant mean curvature $c$ in a neighborhood of every $x\in{\rm reg}\pr\Om\cap\mfR^{n+1}_+$.
	
	On the other hand, notice that \cref{Lem-Cut-off} is applicable here thanks to \cref{Prop-Evg}, and hence
	for any $\ep>0$, we have $\varphi_\ep,S'_\ep$ and $S_\ep$ from \cref{Lem-Cut-off}.
	For any $X\in\mathfrak{X}(\overline{\mfR^{n+1}_+})$,
	let $X_\ep$ be the vector field given by
	\begin{align*}
		X_\ep:=\varphi_\ep X,
	\end{align*}
	since $\varphi_\ep\equiv0$ on $S'_\ep\supset{\rm sing}{\pr\Om}$, we readily see that
	\begin{align*}
		X_\ep
		=\begin{cases}
			0\quad&\text{on }S^{'}_\ep,\\
			\varphi_\ep X\quad&\text{on }S_\ep\setminus S^{'}_\ep,\\
			X\quad&\text{on }\pr\Om\setminus S_\ep.
		\end{cases}
	\end{align*}
	
	Integrating ${\rm div}_{\p\Om}(X_\ep)$ on $\pr\Om\setminus S'_\ep$, since $\pr\Om\setminus S'_\ep$ is regular enough, we may apply the classical tangential divergence theorem to find
	\begin{align*}
		\int_{\pr\Om}{\rm div}_{\p\Om} X_\ep\rd\mcH^n
		=\int_{\pr\Om}X_\ep\cdot(c\nu_\Om)\rd\mcH^n+\int_{\Gamma}X_\ep\cdot \mu\rd\mcH^{n-1}.
	\end{align*}
	A further computation then yields that
	\begin{align*}
		\int_{\pr\Om}\varphi_\ep{\rm div}_{\p\Om}X\rd\mcH^n+\int_{\pr\Om}\nabla^{\pr\Om}\varphi_\ep\cdot X\rd\mcH^n
		=\int_{\pr\Om}\varphi_\ep X\cdot(c\nu_\Om)\rd\mcH^n+\int_\Gamma\varphi_\ep X\cdot\mu\rd\mcH^{n-1}.
	\end{align*}
	
	Since $X\in\mathfrak{X}(\overline{\mfR^{n+1}_+})$ and $\pr\Om$ is bounded, we have that
	${\rm div}_{\p\Om}X=({\rm div}X-\nabla_{\nu_\Om}X\cdot\nu_\Om)$ and $\vert X\vert$ are bounded (the upper bounds are independent of $\ep$).
	By virtue of \eqref{pw-conv} and \eqref{esti-cutoff-1} in \cref{Lem-Cut-off},
	we may send $\ep\searrow0$ and use the dominated convergence theorem to conclude \eqref{formu-div-M}.
	
	On the other hand, since $\p\Om\setminus\pr\Om$ is smooth, and the singularities of $T={\rm cl}_{\p\mfR^{n+1}_+}(\p\Om\setminus\pr\Om)$ are on $\Gamma$,
	we can follow the proof of \eqref{formu-div-M} to conclude \eqref{formu-div-B+}. This completes the proof. 
\end{proof}

\printbibliography

\end{document}